\documentclass[12pt]{amsart}
\usepackage{amsmath,amssymb,amsthm,tikz,color,xcolor,xspace,float}
\usepackage[letterpaper,margin=1in]{geometry}
\usetikzlibrary{decorations.pathreplacing}
\usetikzlibrary{shapes,arrows,svg.path}

\newtheorem{theorem}{\textbf{Theorem}}[section]
\newtheorem{proposition}[theorem]{\textbf{Proposition}}
\newtheorem{corollary}[theorem]{\textbf{Corollary}}
\newtheorem{lemma}[theorem]{\textbf{Lemma}}
\newtheorem{remark}[theorem]{Remark}

\DeclareMathOperator{\GL}{GL}
\DeclareMathOperator{\h}{h}
\let\v\relax
\DeclareMathOperator{\v}{v}
\DeclareMathOperator{\unc}{unc}
\DeclareMathOperator{\col}{col}
\DeclareMathOperator{\mon}{mon}

\newcommand{\plus}{\tikz[anchor=base, baseline=-1pt]{\node[scale=0.8, draw, thick, circle, inner sep=0pt] {$+$};}\xspace}
\newcommand{\minus}{\tikz[anchor=base, baseline=-1pt]{\node[scale=0.8, draw, thick, circle, inner sep=0pt] {$-$};}\xspace}
\newcommand{\boson}[1]{\tikz[anchor=base, baseline=-1pt]{\node[scale=0.8, draw, thick, circle, inner sep=1pt] {$#1$};}\xspace}

\newcommand{\red}{\color{red}}
\newcommand{\blue}{\color{blue}}
\newcommand{\green}{\color{darkgreen}}
\definecolor{darkgreen}{RGB}{0,180,0}
\colorlet[named]{green}{darkgreen}

\begin{document}
\title{Colored Bosonic Models and Matrix Coefficients}
\author{Daniel Bump}
\address{Department of Mathematics, Stanford University, Stanford, CA 94305-2125}
\email{bump@math.stanford.edu}
\author{Slava Naprienko}
\address{Department of Mathematics, UNC Chapel Hill, North Carolina 27599}
\email{slava@naprienko.com}

\begin{abstract}
We develop the theory of colored bosonic models (initiated by
Borodin and Wheeler). We will show how a family of such models
can be used to represent the values of Iwahori vectors in the
``spherical model'' of representations of $\GL_r(F)$, where $F$
is a nonarchimedean local field.
Among our results are a \textit{monochrome factorization}, which is the realization of the Boltzmann weights by fusion of simpler weights, 
a \textit{local lifting} property relating the colored models
with uncolored models, and an action of the Iwahori Hecke algebra on the
partition functions of a particular family of models by 
Demazure-Lusztig operators. 
As an application of the local lifting property we reprove a theorem 
of Korff evaluating the partition functions of the uncolored models in 
terms of Hall-Littlewood polynomials. 
Our results are very closely parallel to the theory of fermionic 
models representing Iwahori Whittaker functions developed by Brubaker,
Buciumas, Bump and Gustafsson, with many striking relationships
between the two theories, confirming the philosophy that the
spherical and Whittaker models of principal series representations
are dual.
\end{abstract}
\maketitle
\tableofcontents

\section{Introduction}

In Brubaker, Buciumas, Bump and Gustafsson~{\cite{BBBGIwahori}} it is proved
that for a basis $\{ \omega_w \}$ of Iwahori Whittaker functions on $G =
\GL_r (F)$ over a nonarchimedean local field, and for all $g \in G$,
there is a solvable lattice model whose partition function equals $\omega_w
(g)$. The lattice models are ``fermionic'' and are associated with the quantum
affine supergroup $U_{\sqrt{q}} (\widehat{\mathfrak{g}\mathfrak{l}} (r| 1))$. On the
other hand, in this paper we will consider bosonic models, associated to the
quantum group $U_{\sqrt{q}} (\widehat{\mathfrak{s}\mathfrak{l}}_{r + 1})$ whose
partition functions are Iwahori spherical matrix coefficients. The two
theories are very closely parallel, as we will now explain.

Before specializing to the case where $G = \GL_r$, let $G$ be a split
reductive algebraic group over a nonarchimedean local field $F$. Let $K$ be a
special maximal compact subgroup. The group $G$ has a family $\{
(\pi_{\mathbf{z}}, I (\mathbf{z})) \}$ of representations, the
{\textit{spherical principal series}}, parametrized by an element $\mathbf{z}$
of a complex torus $\hat{T}$, the maximal torus of the Langland dual group. If
$\mathbf{z}$ is in general position, $\pi_{\mathbf{z}}$ is irreducible.
The space $I (\mathbf{z})$ is infinite dimensional, but we will focus on the
$| W |$-dimensional space $I (\mathbf{z})^J$ of vectors invariant under the
Iwahori subgroup $J$. Here $W$ is the Weyl group. The space $I
(\mathbf{z})^J$ has a standard basis $\phi_w^{\mathbf{z}}$ indexed by
elements $w$ of~$W$.

There are two noteworthy linear functionals on $I (\mathbf{z})$, the
Whittaker functional $\mathcal{W}_{\mathbf{z}}$, and the spherical
functional which we will denote $\mathcal{S}_{\mathbf{z}}$. Both functionals
are unique up to constant multiple, and have natural normalizations as
integrals. We may use them to define ``special functions,'' which are
functions $\omega_w^{\mathbf{z}}$ and $\sigma_w^{\mathbf{z}}$ on $G$
defined by
\begin{equation}
\label{eq:iwahoriws} \omega_w^{\mathbf{z}} (g)
   =\mathcal{W}_{\mathbf{z}} (\pi_{\mathbf{z}} (g) \phi_w^{\mathbf{z}}),
   \qquad \sigma_w^{\mathbf{z}} (g) =\mathcal{S}_{\mathbf{z}}
   (\pi_{\mathbf{z}} (g) \phi_w^{\mathbf{z}}) .
\end{equation}
We will call these {\textit{Iwahori Whittaker functions}} and
{\textit{Iwahori-Spherical functions}}, respectively. The spaces of functions
$\mathcal{W}_{\mathbf{z}} (\pi_{\mathbf{z}} (g) \phi)$ and
$\mathcal{S}_{\mathbf{z}} (\pi_{\mathbf{z}} (g) \phi)$ for $\phi \in I
(\mathbf{z})$ are the {\textit{Whittaker model}} and {\textit{spherical model}}
of the representation $\pi_{\mathbf{z}}$, and the functions
(\ref{eq:iwahoriws}) are the Iwahori-fixed vectors in the models.

There is a significant parallel between these two types of special functions.
It is explained in
{\cite{BrubakerBumpFriedbergMatrix,BBL,BrubakerBumpFriedbergRogawski,BBBF}}
that these two models are related to representations of the (extended) affine Hecke
algebra $\tilde{\mathcal{H}}$. This is the algebra generated by $T_i$ (corresponding to the simple
reflections $s_i \in W$) subject to the quadratic relations
\[ T_i^2 = (q - 1) T_i + q \]
and the braid relations, together with another subalgebra isomorphic to the
group algebra of the weight lattice $\Lambda$ of the Langlands dual group
$\hat{G}$. We will denote by $\mathcal{H}$ the finite-dimensional algebra
generated by the $T_i$, which is a deformation of the group algebra of the
Weyl group.

The two representations of $\tilde{\mathcal{H}}$ in question are called the
{\textit{antispherical}} and {\textit{spherical}} representations, and are induced
from the characters of $\mathcal{H}$ that send $T_i$ to $- 1$ and $q$,
respectively. They can be realized as actions of $\tilde{\mathcal{H}}$ on the
ring $\mathcal{O} (\hat{T})$ of (Laurent) polynomials in $\mathbf{z}$. To
this end, define operators $\mathfrak{T}_i$ and $\mathcal{L}_i$ on
$\mathcal{O} (\hat{T})$ by
\[ \mathfrak{T}_i f (\mathbf{z}) = \frac{f (\mathbf{z}) - f (s_i
   \mathbf{z})}{\mathbf{z}^{\alpha_i} - 1} - q \frac{f (\mathbf{z})
   -\mathbf{z}^{- \alpha_i} f (s_i \mathbf{z})}{\mathbf{z}^{\alpha_i} -
   1} \]
and
\[ \mathcal{L}_i f (\mathbf{z}) = \frac{f (\mathbf{z}) - f (s_i
   \mathbf{z})}{\mathbf{z}^{\alpha_i} - 1} - q \frac{f (\mathbf{z})
   -\mathbf{z}^{\alpha_i} f (s_i \mathbf{z})}{\mathbf{z}^{\alpha_i} - 1}
   . \]
The operators $\mathcal{L}_i$ are called {\textit{Demazure-Lusztig operators}},
and we will call the $\mathfrak{T}_i$ {\textit{Demazure-Whittaker operators}}.
These lead to recursions for the Iwahori Whittaker and spherical functions
({\cite{BrubakerBumpFriedbergMatrix}}), namely
\begin{equation}
  \label{eq:wdemrec} \omega_{s_i w}^{\mathbf{z}} (g) =
  \left\{\begin{array}{ll}
    \mathfrak{T}_i \omega_w^{\mathbf{z}} (g) & \text{if $s_i w > w,$}\\
    \mathfrak{T}_i^{- 1} \omega_w^{\mathbf{z}} (g) & \text{if $s_i w < w$}
  \end{array}\right.
\end{equation}
and similarly
\begin{equation}
  \label{eq:sdemrec} \sigma_{s_i w}^{\mathbf{z}} (g) =
  \left\{\begin{array}{ll}
    \mathcal{L}_i \sigma_w^{\mathbf{z}} (g) & \text{if $s_i w > w,$}\\
    \mathcal{L}_i^{- 1} \sigma_w^{\mathbf{z}} (g) & \text{if $s_i w < w$} .
  \end{array}\right.
\end{equation}
These actions of $\mathcal{H}$ by Demazure-like operators can be extended
to actions of the affine Hecke algebra $\widetilde{H}$ by complementing them with an
action of a large abelian subalgebra isomorphic to the weight lattice.
See~\cite{BrubakerBumpFriedbergMatrix,BrubakerBumpFriedbergRogawski} for this point.
The spherical representation of $\widetilde{H}$, in which the Hecke
generators act by Demazure-Lusztig operators, was shown by 
Lusztig~\cite{LusztigEquivariant} to be an action on the equivariant
K-theory of flag varieties, a fact that was applied by Kazhdan and
Lusztig~\cite{KazhdanLusztigDeligneLanglands} to the Deligne-Lusztig conjecture.

Equation (\ref{eq:wdemrec}) means that if we know the value $\omega_w (g)$ for one $w$, then
obtaining the rest is a matter of applying the $\mathfrak{T}_i$. Now for the
Whittaker model, it was shown in~{\cite{BBBGIwahori}} that there is a $w$ for
which $\omega_w (g)$ is in fact a monomial, but this $w$ depends on $g \in
G$. Specifically, $\omega_w (g)$ depends essentially only on the coset of $g$
in $N \backslash G / J$, where $N$ is a maximal unipotent subgroup, and as
coset representatives we may take $g = \varpi^{\lambda} w_2$ where
$\varpi^{\lambda} = \operatorname{diag} (\varpi^{\lambda_1}, \cdots,
\varpi^{\lambda_r})$, and $\varpi$ is a prime element, and $w_2 \in W$.

\begin{remark}
  \label{rem:basebtu}If $w = w_2$, it is shown in~{\cite{BBBGIwahori}} that
  $\omega_w (g)$ is a monomial, as a consequence of support considerations for
  the Jacquet integral defining the Whittaker function. Furthermore $\omega_w
  (g) = 0$ unless $\lambda$ is {\textit{$w_2$-almost dominant}}, a technical
  condition.
\end{remark}

In Brubaker, Buciumas, Bump and Gustafsson~{\cite{BBBGIwahori}} the following fact is proved. 
It is reproved using a different technique in Naprienko~\cite{NaprienkoWhittaker}.

\begin{theorem}[{\cite{BBBGIwahori}}, Theorem A, \cite{NaprienkoWhittaker}, Section~3.3.1]
  For every $g \in G$, and for every $w \in W$ there is a solvable lattice
  model whose partition function is $\omega_w (g)$.
\end{theorem}

The boundary conditions for the model must encode three pieces of data, 
namely $w, w_2$ and $\lambda$, which is assumed to be $w_2$-almost dominant. 
If $w = w_2$, then the model has a unique state, and its partition
function is a monomial, which may be compared with the base case for the
Iwahori Whittaker functions (Remark~\ref{rem:basebtu}). Then the general case
follows from (\ref{eq:wdemrec}), because the lattice models satisfy the same
Demazure-Whittaker recursion. This is proved using the Yang-Baxter equation.

The purpose of this paper is to show that all of the above features extend to
the spherical model. Our main result (Corollary~\ref{cor:maincor}) is that
for the basis $\sigma_w$ of Iwahori fixed vectors in the spherical model,
and for any $g\in G$ there is a solvable lattice model whose partition
function is $\sigma_w(g)$. 

The models, it turns out, are essentially known, for they are specializations of
models previously investigated by Borodin and Wheeler~{\cite{BorodinWheelerColored}}.
Although they are a special case of these more general models, it is a special case that warrants separate treatment.

The $p$-adic side of the story was
previously investigated by Ion~{\cite{IonNonsymmetric}} (who showed that the
Iwahori-spherical functions are nonsymmetric Macdonald polynomials) and
Brubaker, Bump and Friedberg~{\cite{BrubakerBumpFriedbergRogawski}}.
% See also Casselman~\cite{CasselmanSpherical}, 
% Casselman and Shalika~\cite{CasselmanShalika},
% Rogawski~\cite{RogawskiHecke} and
% Reeder~\cite{ReederIwahori} for related work about the Iwahori fixed
% vectors in the spherical and Whittaker models.
Related work includes Opdam~\cite{OpdamTrace}, 
Cherednik and Ostrik~\cite{CherednikOstrik} Section~10,
Cherednik~\cite{CherednikDFT} and 
Cherednik and Ma~\cite{CherednikMaI,CherednikMaII}
relating $p$-adic matrix coefficients theory with nonsymmetric 
Macdonald polynomials and the~DAHA.

Returning to the comparison of the results of this paper, where
bosonic models represent values of Iwahori vectors in the spherical
model, and those of \cite{BBBGIwahori}, where fermionic models
represent the values of Iwahori Whittaker functions.
An important point for us is to show how remarkably similar the two examples 
are, and also to show how they are different. 

We will describe two kinds of \textit{uncolored} bosonic models, whose
partition functions are, respectively, the Hall-Littlewood
$P$- and $R$-polynomials (\cite{MacdonaldBook} Chapter~III).
The uncolored $P$-models are the same as those in
Korff~\cite{KorffVerlinde}. The $R$-models are close
relatives of the $P$-model, but have more favorable locality
properties. The Hall-Littlewood $P$- and $R$-polynomials
differ by a constant, but that constant
(denoted $v_\lambda(t)$ in~\cite{MacdonaldBook}) depends on $\lambda$,
so this is an important distinction.
The $P$- and $R$-models are related to each other by a simple transformation 
of the Boltzmann weights, called \textit{change of basis} in~\cite{BBB}, Section~4. 
This does not affect the R-matrix, so both models use the same R-matrix
(Figures~\ref{fig:rmatrix} and~\ref{fig:genrmatrix}).

\begin{remark}
The $R$-models are so-called because of their relationship
to the Hall-Littlewood $R$-polynomials. They are not to
be confused with the R-matrices, which appear in the Yang-Baxter
equation.
\end{remark}

After we define the uncolored $R$- and $P$-systems, we will
define \textit{colored} variants, which we will apply to describe matrix
coefficients of principal series representations of 
$\GL_r$ over a nonarchimedean local field.

% Redundant:
% The partition functions of the colored models are nonsymmetric
% Macdonald polynomials, which were previously related to matrix
% coefficients by Ion~\cite{IonNonsymmetric}. The lattice models
% that we employ are special cases of models previously found by
% Borodin and Wheeler~\cite{BorodinWheelerColored}. However 
% this special case (where the spin $s$ of~\cite{BorodinWheelerColored}
% is $0$) has special aspects that warrant giving it a separate treatment. 

These bosonic models are very close analogs of the fermionic
models in Brubaker, Buciumas, Bump and Gustafsson~\cite{BBBGIwahori}
which represent Iwahori Whittaker functions for $\GL_r(F)$, and
we will emphasize these significant parallels. In particular:

\begin{itemize}
    \item The R-matrices in this paper are \textit{identical} to
    those in \cite{BBBGIwahori} except for \textit{one value}
    (Figures~\ref{fig:rmatrix} and~\ref{fig:genrmatrix}). This
    minor change alters the underlying quantum group from
    $U_{\sqrt{q}}(\widehat{\mathfrak{gl}}(r|1))$ (in~\cite{BBBGIwahori}) to
    $U_{\sqrt{q}}(\widehat{\mathfrak{gl}}(r+1))$ (in this paper).
    \item There is a \textit{local lifting property} that relates
    the weights of the uncolored model to the colored model. This
    property also appears in~\cite{BorodinWheelerColored} (Proposition~2.4.2), 
    where it is called \textit{color-blindness}. In
    \cite{BBBGIwahori} the corresponding fact is Properties~A and~B in Section~8.
    In this paper, see Section~\ref{sec:local_lifting}.
    This is most clear in the case of the $R$-models; for the
    $P$-models, it persists in a local lifting property for the
    column transfer matrices.
    \item The vertices and vertical edges admit a factorization
    into \textit{monochrome vertices}. There are thus two versions
    of the model. The vertices of the colored model
    will be called \textit{fused} weights and are obtained from the monochrome
    vertices by \textit{fusion}. This simplifies the proof of the Yang-Baxter
    equation and the local lifting property. We will use fusion in the
    very definition of the Boltzmann weights for the colored model.
\end{itemize}

The solution to the Yang-Baxter RTT equation for the 
bosonic models was found by Kulish~\cite{KulishNonlinear},
in the context of a nonlinear Schr\"odinger difference equation.
It was realized independently by Macfarlane~\cite{MacfarlaneHarmonic} and
Biedenharn~\cite{BiedenharnBoson} that the $q$-harmonic oscillator is
related to the quantum group $U_{\sqrt{q}}({\mathfrak{sl}}_2)$. A Verma module
for this quantum group has a ladder structure like the quantum mechanical
harmonic oscillator, and the different energy levels may be regarded
as states with different numbers of bosons. Hall-Littlewood polynomials
then appeared in Tsilevich~\cite{TsilevichBoson} and Korff~\cite{KorffVerlinde},
who used the explicit bijection between states and tableaux. See
also \cite{BogoliubovIzerginKitanineBoson, BogoliubovBullough, BogoliubovBulloughPangHopping} for related work on bosonic models.

\begin{remark}
The models in this paper are the special case of Borodin
and Wheeler~\cite{BorodinWheelerColored} (1.2.2) with the spin parameter $s=0$.
\end{remark}

For solvable lattice models, there is a paradigm that the edges of the
model each correspond to a module for a quantum group. (For the models
in this paper and~\cite{BBBGIwahori}, this paradigm applies to the
fused versions of the models.) The partition functions in~\cite{KorffVerlinde}
are made with the $U_{\sqrt{q}}(\widehat{\mathfrak{sl}}_2)$ Verma module, and
they are the symmetric Hall-Littlewood polynomials. However in
\cite{BorodinWheelerColored}, colored models based on 
$U_{\sqrt{q}}(\widehat{\mathfrak{sl}}_{r+1})$
Verma modules were introduced, and the resulting partition functions are
nonsymmetric Hall-Littlewood (Macdonald) functions. 
In contrast with the more general models of Borodin and Wheeler (with $s=0$)
the weights in this paper are Verma modules with respect to a maximal parabolic
subgroup whose radical is abelian (before quantization). This may account for the
monochrome factorization property.

The Iwahori spherical functions are special cases of {\textit{matrix
coefficients}} which are functions of the form 
$\langle \pi_{\mathbf{z}} (g) \phi, \psi \rangle$
where $\phi \in I (\mathbf{z})$ and $\psi \in I (-\mathbf{z})$, 
where the pairing $\left\langle \;, \; \right\rangle$ comes
from the fact that $I (-\mathbf{z})$ is the contragredient representation of
$I (\mathbf{z})$. In this paper we are taking $\psi$ to be the $K$-fixed
vector in the contragredient module and $\phi$ to be $J$-fixed. It would
be good to have lattice models whose partition functions are arbitrary
(smooth) matrix coefficients of an arbitrary irreducible representation
of $G$, which would be a full recasting of the representation theory
of $G$. A more modest goal that seems in reach would be to represent
the matrix coefficients for the principal series representations
in which both $\phi$ and $\psi$ are Iwahori-fixed vectors. The results
of this paper, and the
philosophy of Iwahori-Metaplectic duality explained in~\cite{BBBGDuality}
suggests that such models would be bosonic models similar to the fermionic 
models in~\cite{BBBGMetahori} (see also~\cite{ABW}) but for the quantum
group $U_{\sqrt{q}}(\widehat{\mathfrak{sl}}_{2r})$ instead of the super quantum
group $U_{\sqrt{q}}(\widehat{\mathfrak{sl}}(r|n))$. We hope to consider such
models in a later paper.

\medbreak\noindent
\textbf{Acknowledgements:}
We would like to thank Amol Aggarwal, Alexei Borodin, Ben Brubaker, Valentin Buciumas, 
Henrik Gustafsson, Andy Hardt and Christian Korff for helpful conversations about 
bosonic models and the subject of this paper. We are grateful to Valentin Buciumas
and Christian Korff for helpful comments on an earlier draft. We thank
the referees for their careful reading. 

\section{Demazure operators}\label{sec:demazure}
Nothing in this section is new. In~\cite{Demazure}, Demazure introduced operators 
to study line bundles over Schubert varieties,
showed that they satisfy the relations of a degenerate Hecke algebra,
and as a biproduct, obtained a new character formula. In this section we will 
collect the properties we need.

We begin by recalling some properties of Demazure operators. Let $\hat{G}$
be a reductive complex algebraic group, and let $\hat{T}$ be
its maximal torus. Let $\Lambda = X^{\ast} (\hat{T})$ be the
weight lattice of $\hat{G}$, and let $\Phi \subset \Lambda$ be
the root system, with Weyl group $W$ acting on $\Lambda$. 
We partition $\Phi$ into positive and negative roots
$\Phi^{\pm}$, and will denote by $\alpha_1, \cdots, \alpha_r$ the simple
positive roots, and by $s_i$ the simple reflections. An element $\lambda\in\Lambda$
is called \textit{dominant} if $\langle\alpha,\lambda\rangle\geqslant 0$
for positive roots $\alpha$ with respect to a fixed $W$-invariant inner
product $\langle\,,\,\rangle$. Let $w_0$ be the longest
element of $W$. Let $\mathcal{O} (\hat{T})$ be the ring of polynomial
functions on $\hat{T}$. If $\lambda \in \Lambda$ and $\mathbf{z} \in
\hat{T}$ we will denote by $\mathbf{z}^{\lambda}$ the application of
$\lambda$ to $\mathbf{z}$. The functions $\mathbf{z}^{\lambda}$ generate
the ring $\mathcal{O} (\hat{T})$, which is thus isomorphic to the group
algebra of $\Lambda$. Ultimately we will be interested in the case where
$\hat{T}$ is the diagonal torus in $\GL_{r + 1} (\mathbb{C})$ and $\Phi$
is the Type~A root system, but in this section, we may work more generally.

If $f \in \mathcal{O} (\hat{T})$ we will denote
\[ \partial_i f (\mathbf{z}) = \frac{f (\mathbf{z}) -\mathbf{z}^{-
   \alpha_i} f (s_i \mathbf{z})}{1 -\mathbf{z}^{- \alpha_i}}, \]
and
\[ \partial^{\circ}_i f (\mathbf{z}) = \frac{f (\mathbf{z}) - f (s_i
   \mathbf{z})}{\mathbf{z}^{\alpha_i} - 1} . \]
These divided difference operators do not introduce denominators, since the
numerators vanishes when $\mathbf{z}^{\alpha_i} = 1$. It is easy to check
that
\begin{equation}
  \label{demspl} \partial_i = \partial_i^{\circ} + 1.
\end{equation}
\begin{proposition}
  The operators $\partial_i$ and $\partial_i^{\circ}$ satisfy the braid
  relations for the Weyl group $W$. Consequently if $w = s_{i_1} \cdots
  s_{i_k}$ is a reduced expression we may define
  \[ \partial_w f = \partial_{i_1} \cdots \partial_{i_k} f, \qquad
     \partial^{\circ}_w f = \partial^{\circ}_{i_1} \cdots
     \partial^{\circ}_{i_k} f. \]
\end{proposition}

Thus if $D_i$ denotes either $\partial_i$ or $\partial_i^{\circ}$, and if $1
\leqslant i, j \leqslant r$, then the braid relation asserts that
\[ D_i D_j D_i \cdots \; = \; D_j D_i D_j \cdots \]
where the number of terms on both sides is the order of $s_i s_j$ in $W$.

\begin{proof}
  This is proved in Chapter~25 of~{\cite{BumpLie}}, where $\partial_i^{\circ}$
  is denoted $D_i$. For $\partial_i^{\circ}$, this is Proposition~25.1, while
  for $\partial_i$ it is Proposition~25.3. (There is a typo, where the wrong
  font $D_i$ is used instead of $\partial_i$ for the statement of this
  proposition.) The implication that $\partial_w$ and $\partial_w^{\circ}$ are
  well-defined as a consequence of the braid relations is ``Matsumoto's
  theorem,'' Theorem~25.2 in~{\cite{BumpLie}}.
\end{proof}

Let $\rho = \frac{1}{2}\sum_{\alpha \in \Phi^+}\alpha$ be the Weyl vector.
If $f \in \mathcal{O} (\hat{T})$, define
\[ \Omega (f) = \prod_{\alpha \in \Phi^+} (1 -\mathbf{z}^{- \alpha})^{- 1}
   \mathbf{z}^{- \rho} \sum_{w \in W} (- 1)^{\ell (w)} w
   (\mathbf{z}^{\rho} f) . \]
If $\lambda\in\Lambda$ is dominant, then by the Weyl character formula 
$\Omega (\mathbf{z}^{\lambda}) =
\chi_{\lambda} (\mathbf{z})$, where $\chi_{\lambda}
(\mathbf{z})$ is the character of the irreducible representation of $\hat{G}$ with
highest weight $\lambda$. In the Type $A$ case, if $\lambda$ is a partition
(which is a dominant weight) this is just the Schur
polynomial~$s_{\lambda} (\mathbf{z})$.

The next result makes use of the strong Bruhat order $\leqslant$ on $W$.

\begin{proposition}
  \label{prop:dembru}Let $w \in W$. Then
  \[ \partial_w = \sum_{y \leqslant w} \partial_y^{\circ} . \]
\end{proposition}

\begin{proof}
  See~{\cite{demice}}, Theorem~2.1.
\end{proof}

\begin{proposition}[Demazure]
  \label{prop:demomega}
  We have
  \[ \Omega = \partial_{w_0}, \qquad \Omega = \sum_{w \in W}
     \partial_w^{\circ} . \]
\end{proposition}

\begin{proof}
  By Theorem~25.3 in {\cite{BumpLie}} we have $\Omega = \partial_{w_0}$.
  The identity $\Omega = \sum_{w \in W} \partial_w^{\circ}$ follows from
  Proposition~\ref{prop:dembru} with $w = w_0$.
\end{proof}

\begin{lemma}
  \label{lem:omegacases}We have
  \[ \label{eq:lemcasa} \Omega (\mathbf{z}^{w (\lambda + \rho) - \rho}) = (-
     1)^{\ell (w)} \Omega (\mathbf{z}^{\lambda}) . \]
  As a particular case
  \[ \label{eq:lemcasb} \Omega (\mathbf{z}^{w_0 \lambda - 2 \rho}) = (-
     1)^{\ell (w_0)} \Omega (\mathbf{z}^{\lambda}) . \]
\end{lemma}

\begin{proof}
  We will use the dot action of the Weyl group on $\Lambda$:
  \[ w \bullet \lambda = w (\lambda + \rho) - \rho . \]
  We may then write
  \[ \Omega (\mathbf{z}^{\lambda}) = \prod_{\alpha \in \Phi^+} (1
     -\mathbf{z}^{- \alpha})^{- 1} \sum_{w \in W} (- 1)^{\ell (w)}
     \mathbf{z}^{w \bullet \lambda} \]
  and from this the statement follows from a change of variables, using the
  fact that $\bullet$ is a group action.
\end{proof}

Now let $q$ be a parameter, which can be an element of $\mathbb{C}^{\times}$
or an indeterminant. If $\lambda\in\Lambda$ is a dominant weight, define
\begin{equation}
\label{eq:rpoldef}
R_{\lambda} (\mathbf{z}; q) = \Omega \!\left( \prod_{\alpha \in \Phi^{_+}}
   (1 - q\mathbf{z}^{- \alpha}) \mathbf{z}^{\lambda} \right) .
\end{equation}
It is easy to show that
\begin{equation}
\label{eq:rpolequiv}
R_{\lambda} (\mathbf{z}; q) = \sum_{w \in W} w \!\left( \prod_{\alpha \in
   \Phi^+} \frac{1 - q\mathbf{z}^{- \alpha}}{1 -\mathbf{z}^{- \alpha}}
   \mathbf{z}^{\lambda} \right) . 
\end{equation}   
In the Type A case, if $\lambda$ is a partition, these are the Hall-Littlewood 
$R$-polynomials (Macdonald~{\cite{MacdonaldBook}}, Chapter~3). We will avoid
using the notation $R_\lambda$ if $\lambda$ is not dominant, even though expressions
such as the right hand side of (\ref{eq:rpoldef}) will arise. A special case
is given by the following result.
\begin{proposition}
  \label{prop:rlamns}If $\lambda$ is dominant, then
  \[ \Omega \left( \prod_{\alpha \in \Phi^{_+}} (1 - q\mathbf{z}^{- \alpha})
     \mathbf{z}^{w_0 \lambda} \right) = q^{| \Phi^+ |} R_{\lambda}
     (\mathbf{z}; q^{- 1}) . \]
\end{proposition}

\begin{proof}
  Expanding the product gives the sum over subsets of $\Phi^+$:
  \[ \sum_{S \subseteq \Phi^+} (- q)^{| S |} \Omega \left( \mathbf{z}^{w_0
     \lambda - \sum_{\alpha \in S} \alpha} \right) . \]
  Replacing $S$ by $T = \Phi^+ - S$ this equals
  \[ (- q)^{| \Phi^+ |} \sum_{T \subseteq \Phi^+} (- q)^{- | T |} \Omega
     \left( \mathbf{z}^{w_0 \lambda - 2 \rho + \sum_{\alpha \in T} \alpha}
     \right). \]
  Let $U = - w_0 T$. Then
  \[ \sum_{\alpha \in T} \alpha = - w_0 \left( \sum_{\alpha \in U} \alpha
     \right) . \]
  The map $T \mapsto - w_0 T$ permutes the subsets of $\Phi^+$ and using 
  Lemma~\ref{lem:omegacases} we obtain
  \[ (- q)^{| \Phi^+ |} \sum_{U \subseteq \Phi^+} (- q)^{- | U |} \Omega
     \left( \mathbf{z}^{w_0 \left( \lambda - \sum_{\alpha \in U} \alpha
     \right) - 2 \rho} \right) = q^{| \Phi^+ |} \sum_{U \subseteq \Phi^+} (-
     q)^{- | U |} \Omega \left( \mathbf{z}^{\lambda - \sum_{\alpha \in U}
     \alpha} \right) . \]
  This equals
  \[ q^{| \Phi^+ |} \Omega \left( \prod_{\alpha \in \Phi^+} (1 - q^{- 1}
     \mathbf{z}^{- \alpha}) \mathbf{z}^{\lambda} \right) = q^{| \Phi^+ |}
     R_{\lambda} (\mathbf{z}; q^{- 1}) . \qedhere \]
  
\end{proof}

Let $\mathcal{O}_q(\hat{T})=\mathcal{O}(\hat{T})[q,q^{-1}]$. If we regard $q$
as an indeterminate, this is a Laurent polynomial ring in $q$ over $\mathcal{O}(\hat{T})$. 
For $f \in \mathcal{O}_q(\hat{T})$, the Demazure-Lusztig operators are defined by
\begin{equation}
\label{eq:dldef}
\mathcal{L}_{i, q} f (\mathbf{z}) =\mathcal{L}_{i, q}^{\mathbf{z}} f
   (\mathbf{z}) = \frac{f (\mathbf{z}) - f (s_i
   \mathbf{z})}{\mathbf{z}^{\alpha_i} - 1} - q \frac{f (\mathbf{z})
   -\mathbf{z}^{\alpha_i} f (s_i \mathbf{z})}{\mathbf{z}^{\alpha_i} - 1} . 
\end{equation}
They map $\mathcal{O}(\hat{T})$ into $\mathcal{O}_q(\hat{T})$.
They satisfy the quadratic relations 
\[\mathcal{L}_{i, q}^2 = (q - 1) \mathcal{L}_{i, q} + q.\]
The inverse operator equals
\begin{equation}
\label{eq:dlinverse}
\mathcal{L}_{i,q}^{- 1} f (\mathbf{z}) = \left( \frac{f (\mathbf{z}) - f
   (s_i \mathbf{z})}{\mathbf{z}^{- \alpha_i}-1} \right) - q^{- 1} \left(
   \frac{f (\mathbf{z}) -\mathbf{z}^{- \alpha_i} f (s_i \mathbf{z})}
   {\mathbf{z}^{- \alpha_i}-1} \right).
\end{equation}
The $\mathcal{L}_i$ satisfy the braid relations by
Lusztig~{\cite{LusztigEquivariant}} equation (5.2), proved in Section~8.
Consequently if $w = s_{i_1} \cdots s_{i_k}$ is a reduced expression for $w
\in W$, we may write
\[ \mathcal{L}_{w, q} =\mathcal{L}^{\mathbf{z}}_{w, q} =\mathcal{L}_{i_1, q}
   \cdots \mathcal{L}_{i_k, q}, \]
and this is well-defined.

\begin{lemma}\label{lem:demluszforms}
We have
\begin{equation}
\label{eq:lidemazure} 
\mathcal{L}_{i,q}+1=q(\mathcal{L}_{i,q}^{-1}+1)=\partial_i(1-q\mathbf{z}^{-\alpha_i}).
\end{equation}
\end{lemma}

\begin{proof}
Each operator in (\ref{eq:lidemazure}) may be computed to equal
\[(\mathbf{z}^{\alpha_i}-1)^{-1}(\mathbf{z}^{\alpha_i}-s_i-q+q\mathbf{z}^{\alpha_i}s_i).
\qedhere\]
\end{proof}

\begin{proposition}
  \label{prop:sphericaldem}Let $f \in \mathcal{O}(\hat{T})$. Then
  \[ \sum_{w \in W} \mathcal{L}_{w, q} f = \Omega \left( \prod_{\alpha \in \Phi^+} (1 -
     q\mathbf{z}^{- \alpha}) f \right) . \]
  In particular, if $\lambda$ is dominant,
  \[ \sum_{w \in W} \mathcal{L}_{w, q} \mathbf{z}^\lambda =
    R_\lambda(\mathbf{z}; q).\]
  \end{proposition}

\begin{proof}
  See~{\cite{BrubakerBumpFriedbergMatrix}}, Theorem~14 for a statement that
  includes this fact. For completeness, we give a proof here.
  Let $\Theta$ be the operator
  \begin{equation}
    \label{eq:Thetadef} \Theta = \sum_{w \in W} \mathcal{L}_{w, q}.
  \end{equation}
  First we argue that if $w \in W$ then $w \Theta = \Theta$. It is enough to
  check this when $w = s_i$ is a simple reflection. We write
  \[ \Theta = (1 +\mathcal{L}_{i, q}) 
  \sum_{\substack{ w \in W\\ s_i w > w }} 
  \mathcal{L}_{w, q} = \partial_i (1 - q\mathbf{z}^{- \alpha_i}) 
  \sum_{\substack{ w \in W\\ s_i w > w }} 
  \mathcal{L}_{w, q}, \]
  and since $s_i \partial_i = \partial_i$ (as is easily checked) the statement
  follows. On the other hand let
  \[ \Theta' = \Omega \prod_{\alpha \in \Phi^+} (1 - q\mathbf{z}^{- \alpha}) .  \]
  We wish to show $\Theta = \Theta'$. Since $w \Omega = \Omega$ we also have $w
  \Theta' = \Theta'$.

  We may expand $\Theta$ and $\Theta'$ and write
  \[ \Theta = \sum_{w \in W} h_{w, q} (\mathbf{z}) \, w, \qquad \Theta' =
     \sum_{w \in W} h_{w, q}' (\mathbf{z}) \, w. \]
  If $y \in W$ then since $y \Theta = \Theta$ we must have
  \[ y (h_{w, q} (\mathbf{z})) = h_{y w, q} (\mathbf{z}), \]
  and similarly for $h'_{w, q}$. Thus to show $\Theta = \Theta'$ we need only
  show $h_{w, q} = h'_{w, q}$ for one $w$. We will prove this when $w = w_0$.
  
  Only $\mathcal{L}_{w_0, q}$ can contribute to the coefficient of $w_0$. Let
  $w_0 = s_{i_1} \cdots s_{i_N}$ be a reduced expression. We write
  \[ \mathcal{L}_{w_0, q} =\mathcal{L}_{i_1, q} \cdots \mathcal{L}_{i_N, q} =
     \prod_{j = 1}^N \left( \left( \frac{1 - q}{\mathbf{z}^{\alpha_{i_j}} - 1}
     \right) + \frac{1 - q\mathbf{z}^{\alpha_{i_j}}}{1
     -\mathbf{z}^{\alpha_{i_j}}} s_{i_j} \right) . \]
  In multiplying this out, the only way to get $w_0$ is to take the term
  involving $s_{i_j}$ from each factor and therefore
  \[ h_{w_0, q} (\mathbf{z}) = \left(  \frac{1 -
     q\mathbf{z}^{\alpha_{i_1}}}{1 -\mathbf{z}^{\alpha_{i_1}}} s_{i_1}
     \right) \cdots \left(  \frac{1 - q\mathbf{z}^{\alpha_{i_N}}}{1
     -\mathbf{z}^{\alpha_{i_N}}} s_{i_N} \right) . \]
  Moving the $s_{i_j}$ to the right,
  \[ \prod_{j = 1}^N \left(  \frac{1 - q\mathbf{z}^{\beta_1}}{1
     -\mathbf{z}^{\beta_1}} \right) \cdots \left( \frac{1 -
     q\mathbf{z}^{\beta_N}}{1 -\mathbf{z}^{\beta_N}} \right) s_{i_1} \cdots
     s_{i_N}, \]
  where $\beta_1 = \alpha_{i_1}$, $\beta_2 = s_{i_1} (\alpha_{i_2})$, $\beta_3 =
  s_{i_1} s_{i_2} (\alpha_{i_3}), \cdots$. It is well-known that $\beta_1,
  \cdots, \beta_N$ is an enumeration of the positive roots ({\cite{BumpLie}}
  Proposition 20.10), so
  \[ h_{w_0, q} (\mathbf{z}) = \prod_{\alpha \in \Phi^+} \left( \frac{1 -
     q\mathbf{z}^{\alpha}}{1 -\mathbf{z}^{\alpha}} \right) = w_0 \left(
     \prod_{\alpha \in \Phi^+} \left( \frac{1 - q\mathbf{z}^{- \alpha}}{1
     -\mathbf{z}^{- \alpha}} \right) \right) \]
  which also equals $h'_{w_0, q}$.
\end{proof}

\begin{theorem}
  \label{thm:taurecurse}
  Let $y \in W$ and let $\lambda \in \Lambda$. Then there exists a unique family of
  functions $\tau^{\lambda}_{w, y} (\mathbf{z};q) \in \mathcal{O}_q (\hat{T})$
  indexed by $w \in W$ such that if $s = s_i$ is a simple reflection
  \begin{equation}
    \label{eq:sigmarecurse} \tau^{\lambda}_{s w, y} (\mathbf{z};q) =
    \left\{\begin{array}{ll}
      \mathcal{L}_{i, q} \tau^{\lambda}_{w, y}(\mathbf{z};q) & \text{if $s w > w$,}\\
      \mathcal{L}_{i, q}^{- 1} \tau^{\lambda}_{w, y}(\mathbf{z};q) & \text{if $s w < w$,}
    \end{array}\right.
  \end{equation}
  and such that $\tau^{\lambda}_{y, y}(\mathbf{z};q) = q^{\ell (y)} \mathbf{z}^{\lambda}$. 
  Moreover, we have
  \[
  	\sum_{w \in W}\tau^{\lambda}_{w, y} (\mathbf{z};q) 
    = \Omega\left(\prod_{\alpha \in \Phi^+}(1-qz^{-\alpha})\mathbf{z}^\lambda \right).
  \]
  In particular, if $\lambda$ is dominant, then
  \begin{equation}
    \label{eq:genrform} \sum_{w \in W} \tau^{\lambda}_{w, y} (\mathbf{z};q) =
    R_{\lambda} (\mathbf{z}; q),
  \end{equation}
  and if $\lambda$ is antidominant, then
  \begin{equation}
    \label{eq:dualrform} \sum_{w \in W} \tau^{\lambda}_{w, y} (\mathbf{z};q) =
    q^{|\Phi^+|}R_{w_0\lambda} (\mathbf{z}; q^{-1}) .
  \end{equation}
	The functions $\tau^\lambda_{w,y}$ are given by
	\begin{equation}
		\label{eq:taudef}
	\tau^{\lambda}_{w, y} (\mathbf{z};q) = q^{\ell (y)} \mathcal{L}_{w, q} \mathcal{L}_{y, q}^{- 1} \mathbf{z}^{\lambda}.
	\end{equation}
  Furthermore, we have
  \begin{equation}
  \label{eq:tauinvolution}
  \tau^{\lambda}_{w, y} (\mathbf{z}; q) = q^{| \Phi^+ |} 
  \tau^{-\lambda}_{w w_0, y w_0} (\mathbf{z}^{- 1} ; q^{- 1}) .
  \end{equation}

\end{theorem}

Thus summing over $w$ in (\ref{eq:genrform}) produces \textit{two}
$W$ invariance properties that $\tau^\lambda_{w,y}(\mathbf{z})$ does
not possess: the sum becomes symmetric, invariant under the action of $W$
on $\mathbf{z}$, and it also becomes independent on the Weyl group
element $y$. 

\begin{proof}
	We may define $\tau^\lambda_{w,y}$ by (\ref{eq:taudef})
	and the identity
  (\ref{eq:sigmarecurse}) follows easily. The problem will be to prove
  (\ref{eq:genrform}) and (\ref{eq:dualrform}). Both are cases of the
  identity
  \begin{equation}
  \label{eq:gensumid}
  \sum_{w \in W} \tau^{\lambda}_{w, y} (\mathbf{z};q) = \Omega \left(\prod_{\alpha \in \Phi^+} (1
     - q\mathbf{z}^{- \alpha}) \mathbf{z}^{\lambda} \right),
  \end{equation}
  which we will prove for any $\lambda$. If $\lambda$ is dominant, then
  (\ref{eq:gensumid}) implies (\ref{eq:genrform}) from the definition of
  $R_\lambda$, while if $\lambda$ is antidominant, then it implies
  (\ref{eq:dualrform}) by Proposition~\ref{prop:rlamns}.

  If $y = 1$, then Proposition~\ref{prop:sphericaldem}
  implies (\ref{eq:gensumid}) in this case. To prove the more general case it
  suffices to show that
  \[ q^{\ell (y)} \Theta \mathcal{L}_{y, q}^{- 1} = \Theta \]
  where $\Theta$ is defined by (\ref{eq:Thetadef}). This reduces immediately
  to the case where $y = s_i$ is a simple reflection. Then we have
  \[ q \Theta \mathcal{L}_{i, q}^{- 1} = q \left(
     \sum_{\substack{ w \in W\\ w s_i > w }} 
     \mathcal{L}_{w, q} \right) (1 +\mathcal{L}_{i, q})
     \mathcal{L}_{i,q}^{- 1} = q \left( \sum_{\substack{ w \in W\\ w s_i > w }} 
     \mathcal{L}_{w, q} \right) (1 +\mathcal{L}_{i, q}^{- 1}) .
  \]
  Using (\ref{eq:lidemazure}), this equals
  \[ \left( \sum_{\substack{ w \in W\\ w s_i > w }} 
  \mathcal{L}_{w, q} \right) (1 +\mathcal{L}_{i, q}) = \Theta
     . \]
     
  To prove the involution identity (\ref{eq:tauinvolution}), 
  define by $c:\mathcal{O}(\hat{T})\to\mathcal{O}(\hat{T})$
  the map $cf(\mathbf{z})=f(\mathbf{z}^{-1})$. Comparing
  (\ref{eq:dldef}) and (\ref{eq:dlinverse}) shows that
  \begin{equation}
  \mathcal{L}_{i,q}^{-1}=c\,\mathcal{L}_{i,q^{-1}}\,c.
  \end{equation}
  Using this identity, together with the fact that
  the map $w\mapsto ww_0$ is order-reversing for the Bruhat
  order, it is easy to check that the right-hand side of
  (\ref{eq:tauinvolution}) satisfies the defining properties
  of~$\tau^{\lambda}_{w,y}$.
\end{proof}

We note that functions $\tau_{w,y}^\lambda(\mathbf{z};q)$ are closely related 
to the permuted basement non-symmetric Hall-Littlewood polynomials (as
specialization of the corresponding permuted basement non-symmetric Macdonald
polynomials) introduced in~\cite{FerreiraThesis}. See also
\cite{Al16, GuoRam, CorteelMandelshtamWilliamsQueues, Alexandersson2020}. 

In this paper, we show that the $\tau_{w,y}^\lambda(\mathbf{z};q)$ are equal
to the Iwahori-spherical matrix coefficients in the standard basis of the 
Iwahori-fixed vectors. 

\section{\label{sec:uncolored}The uncolored bosonic models}
We recall some basics about solvable lattice models. A lattice model is a combinatorial system $\mathfrak{S}$ with an ensemble of states
$\mathfrak{s}$. For every state $\mathfrak{s}\in\mathfrak{S}$
there is assigned a \textit{Boltzmann weight} $\beta(\mathfrak{s})$.
The \textit{partition function} is then
\begin{equation}
\label{eq:parfun}
Z(\mathfrak{S})=\sum_{\mathfrak{s}\in\mathfrak{S}}\beta(\mathfrak{s}).
\end{equation}

To specify the system $\mathfrak{S}$ we begin
with a planar array consisting of vertices and edges, such as
an $r\times N$ grid. We will assume that each vertex is adjacent to four edges.
The edges are divided into \textit{boundary edges} at the boundary of the
array, and \textit{interior edges} in the interior. Every interior edge is adjacent to two vertices, and every boundary edge is adjacent to a single vertex. Every edge $E$ is assigned a \textit{spinset} 
$\Sigma_E$ of allowable states. Every boundary edge is assigned a fixed
element of its spinset, which is fixed and is part of the data defining
the system. To specify a state $\mathfrak{s}\in\mathfrak{S}$, we assign
spins to the internal edges, picking for each edge $E$ a spin $s\in\Sigma_E$.
Now as part of the data specifying the system $\mathfrak{S}$, for each
vertex $v$ we assign a set of \textit{local Boltzmann weights},
which is a function from the set of possible spins of its four
adjacent edges to $\mathbb{C}$. Thus given a state $\mathfrak{s}$ every
vertex is assigned a local Boltzmann weight, and the product of these
over all vertices is the Boltzmann weight $\beta(\mathfrak{s})$ appearing
in (\ref{eq:parfun}). We say that a vertex $v$ is admissible if its Boltzmann weight is non-zero. We say that a state $\mathfrak{s}$ is admissible, if all vertices $v \in \mathfrak{s}$ are admissible. 
%% Figure 1
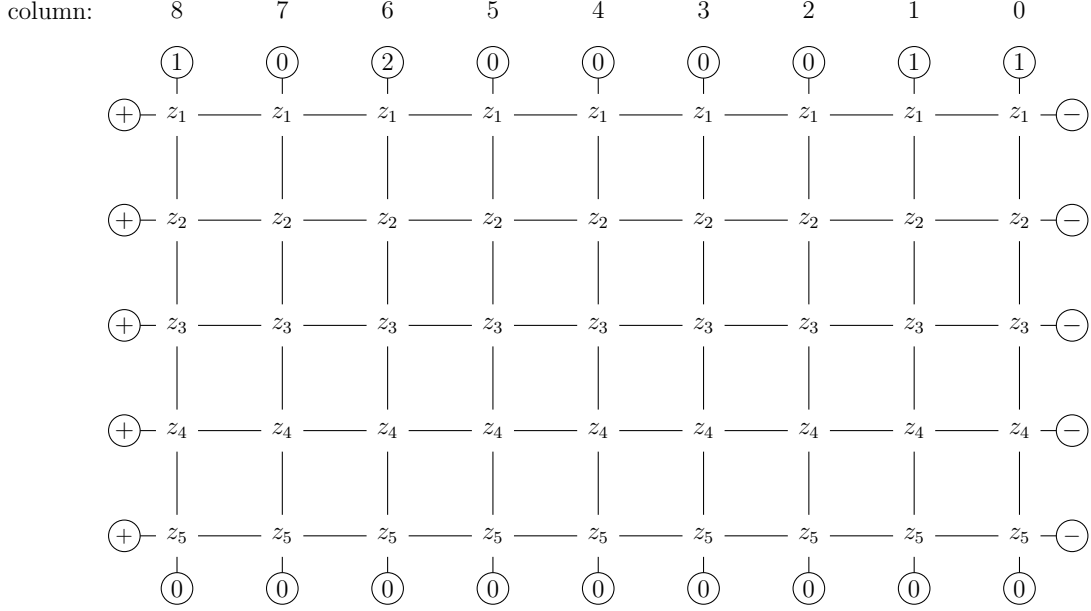
\begin{figure}
\[\begin{tikzpicture}[scale=.7,every node/.style={scale=.8}]
\draw (0,1)--(18,1);
\draw (0,3)--(18,3);
\draw (0,5)--(18,5);
\draw (0,7)--(18,7);
\draw (0,9)--(18,9);
\draw (1,0)--(1,10);
\draw (3,0)--(3,10);
\draw (5,0)--(5,10);
\draw (7,0)--(7,10);
\draw (9,0)--(9,10);
\draw (11,0)--(11,10);
\draw (13,0)--(13,10);
\draw (15,0)--(15,10);
\draw (17,0)--(17,10);
\path[fill=white] (1,9) circle (.4);
\path[fill=white] (3,9) circle (.4);
\path[fill=white] (5,9) circle (.4);
\path[fill=white] (7,9) circle (.4);
\path[fill=white] (9,9) circle (.4);
\path[fill=white] (11,9) circle (.4);
\path[fill=white] (13,9) circle (.4);
\path[fill=white] (15,9) circle (.4);
\path[fill=white] (17,9) circle (.4);
\node at (1,9) {$z_1$};
\node at (3,9) {$z_1$};
\node at (5,9) {$z_1$};
\node at (7,9) {$z_1$};
\node at (9,9) {$z_1$};
\node at (11,9) {$z_1$};
\node at (13,9) {$z_1$};
\node at (15,9) {$z_1$};
\node at (17,9) {$z_1$};
\path[fill=white] (1,7) circle (.4);
\path[fill=white] (3,7) circle (.4);
\path[fill=white] (5,7) circle (.4);
\path[fill=white] (7,7) circle (.4);
\path[fill=white] (9,7) circle (.4);
\path[fill=white] (11,7) circle (.4);
\path[fill=white] (13,7) circle (.4);
\path[fill=white] (15,7) circle (.4);
\path[fill=white] (17,7) circle (.4);
\node at (1,7) {$z_2$};
\node at (3,7) {$z_2$};
\node at (5,7) {$z_2$};
\node at (7,7) {$z_2$};
\node at (9,7) {$z_2$};
\node at (11,7) {$z_2$};
\node at (13,7) {$z_2$};
\node at (15,7) {$z_2$};
\node at (17,7) {$z_2$};
\path[fill=white] (1,5) circle (.4);
\path[fill=white] (3,5) circle (.4);
\path[fill=white] (5,5) circle (.4);
\path[fill=white] (7,5) circle (.4);
\path[fill=white] (9,5) circle (.4);
\path[fill=white] (11,5) circle (.4);
\path[fill=white] (13,5) circle (.4);
\path[fill=white] (15,5) circle (.4);
\path[fill=white] (17,5) circle (.4);
\node at (1,5) {$z_3$};
\node at (3,5) {$z_3$};
\node at (5,5) {$z_3$};
\node at (7,5) {$z_3$};
\node at (9,5) {$z_3$};
\node at (11,5) {$z_3$};
\node at (13,5) {$z_3$};
\node at (15,5) {$z_3$};
\node at (17,5) {$z_3$};
\path[fill=white] (1,3) circle (.4);
\path[fill=white] (3,3) circle (.4);
\path[fill=white] (5,3) circle (.4);
\path[fill=white] (7,3) circle (.4);
\path[fill=white] (9,3) circle (.4);
\path[fill=white] (11,3) circle (.4);
\path[fill=white] (13,3) circle (.4);
\path[fill=white] (15,3) circle (.4);
\path[fill=white] (17,3) circle (.4);
\node at (1,3) {$z_4$};
\node at (3,3) {$z_4$};
\node at (5,3) {$z_4$};
\node at (7,3) {$z_4$};
\node at (9,3) {$z_4$};
\node at (11,3) {$z_4$};
\node at (13,3) {$z_4$};
\node at (15,3) {$z_4$};
\node at (17,3) {$z_4$};
\path[fill=white] (1,1) circle (.4);
\path[fill=white] (3,1) circle (.4);
\path[fill=white] (5,1) circle (.4);
\path[fill=white] (7,1) circle (.4);
\path[fill=white] (9,1) circle (.4);
\path[fill=white] (11,1) circle (.4);
\path[fill=white] (13,1) circle (.4);
\path[fill=white] (15,1) circle (.4);
\path[fill=white] (17,1) circle (.4);
\node at (1,1) {$z_5$};
\node at (3,1) {$z_5$};
\node at (5,1) {$z_5$};
\node at (7,1) {$z_5$};
\node at (9,1) {$z_5$};
\node at (11,1) {$z_5$};
\node at (13,1) {$z_5$};
\node at (15,1) {$z_5$};
\node at (17,1) {$z_5$};
\draw[fill=white] (0,1) circle (.3);
\node[scale=.9] at (0,1) {$+$};
\draw[fill=white] (0,3) circle (.3);
\node at (0,3) {$+$};
\draw[fill=white] (0,5) circle (.3);
\node at (0,5) {$+$};
\draw[fill=white] (0,7) circle (.3);
\node at (0,7) {$+$};
\draw[fill=white] (0,9) circle (.3);
\node at (0,9) {$+$};
\draw[fill=white] (18,1) circle (.3);
\node[scale=.9] at (18,1) {$-$};
\draw[fill=white] (18,3) circle (.3);
\node at (18,3) {$-$};
\draw[fill=white] (18,5) circle (.3);
\node at (18,5) {$-$};
\draw[fill=white] (18,7) circle (.3);
\node at (18,7) {$-$};
\draw[fill=white] (18,9) circle (.3);
\node at (18,9) {$-$};
\draw[fill=white] (1,10) circle (.3);
\node at (1,10) {$1$};
\draw[fill=white] (3,10) circle (.3);
\node at (3,10) {$0$};
\draw[fill=white] (5,10) circle (.3);
\node at (5,10) {$2$};
\draw[fill=white] (7,10) circle (.3);
\node at (7,10) {$0$};
\draw[fill=white] (9,10) circle (.3);
\node at (9,10) {$0$};
\draw[fill=white] (11,10) circle (.3);
\node at (11,10) {$0$};
\draw[fill=white] (13,10) circle (.3);
\node at (13,10) {$0$};
\draw[fill=white] (15,10) circle (.3);
\node at (15,10) {$1$};
\draw[fill=white] (17,10) circle (.3);
\node at (17,10) {$1$};
\draw[fill=white] (1,0) circle (.3);
\node at (1,0) {$0$};
\draw[fill=white] (3,0) circle (.3);
\node at (3,0) {$0$};
\draw[fill=white] (5,0) circle (.3);
\node at (5,0) {$0$};
\draw[fill=white] (7,0) circle (.3);
\node at (7,0) {$0$};
\draw[fill=white] (9,0) circle (.3);
\node at (9,0) {$0$};
\draw[fill=white] (11,0) circle (.3);
\node at (11,0) {$0$};
\draw[fill=white] (13,0) circle (.3);
\node at (13,0) {$0$};
\draw[fill=white] (15,0) circle (.3);
\node at (15,0) {$0$};
\draw[fill=white] (17,0) circle (.3);
\node at (17,0) {$0$};
\node at (-1.4,11) {column:};
\node at (1,11) {$8$};
\node at (3,11) {$7$};
\node at (5,11) {$6$};
\node at (7,11) {$5$};
\node at (9,11) {$4$};
\node at (11,11) {$3$};
\node at (13,11) {$2$};
\node at (15,11) {$1$};
\node at (17,11) {$0$};
\end{tikzpicture}\]
\caption{The grid with boundary conditions for the uncolored model, 
$\mathfrak{S}^P_\lambda(\mathbf{z};t)$ or $\mathfrak{S}^R_\lambda(\mathbf{z}; t)$,
corresponding to the partition $\lambda=(8,6,6,1,0)$, with $r=5$. A state of the model will assign spins to the interior edges.}
\label{fig:uncolored_model}
\end{figure}

%% Figure 2
% \input{figures/uncolored_model}%
% \input{figures/uncolored_weights}%
\begin{figure}[tb]
\[\begin{array}{|c|c|c|c|c|}
\hline 
\text{name} & A(n) & B(n) & C(n) & D(n) \\
\hline
& \begin{tikzpicture}[scale=.8,every node/.style={scale=0.85}]
% \draw[line width=0.5pt, decoration={brace},decorate] (0.2,7) -- node[above=6pt] {$\col$} (6.2,7);
\draw (-1,0) to (1,0);
\draw (0,-1) to (0,1);
\path[fill=white] (0,0) circle (.25);
\draw[fill=white] (1,0) circle (.35);
\draw[fill=white] (-1,0) circle (.35);
\draw[fill=white] (0,1) circle (.35);
\draw[fill=white] (0,-1) circle (.35);
\node at (0,0) {$z_i$};
\node at (-1,0) {$+$};
\node at (1,0) {$+$};
\node at (0,-1) {$n$};
\node at (0,1) {$n$};
\end{tikzpicture} &
\begin{tikzpicture}[scale=.8,every node/.style={scale=0.85}]
\draw (-1,0) to (1,0);
\draw (0,-1) to (0,1);
\path[fill=white] (0,0) circle (.25);
\draw[fill=white] (1,0) circle (.35);
\draw[fill=white] (-1,0) circle (.35);
\draw[fill=white] (0,1) circle (.35);
\draw[fill=white] (0,-1) circle (.35);
\node at (0,0) {$z_i$};
\node at (-1,0) {$-$};
\node at (1,0) {$-$};
\node at (0,-1) {$n$};
\node at (0,1) {$n$};
\end{tikzpicture} &
\begin{tikzpicture}[scale=.8,every node/.style={scale=0.85}]
\draw (-1,0) to (1,0);
\draw (0,-1) to (0,1);
\path[fill=white] (0,0) circle (.35);
\draw[fill=white] (1,0) circle (.35);
\draw[fill=white] (-1,0) circle (.35);
\draw[fill=white] (0,1) circle (.35);
\draw[fill=white] (0,-1) circle (.35);
\node at (0,0) {$z_i$};
\node at (-1,0) {$-$};
\node at (1,0) {$+$};
\node at (0,1) {$n$};
\node at (0,-1) [scale=.75]{$n\!\!+\!\!1$};
\end{tikzpicture} &
\begin{tikzpicture}[scale=.8,every node/.style={scale=0.85}]
\draw (-1,0) to (1,0);
\draw (0,-1) to (0,1);
\path[fill=white] (0,0) circle (.35);
\draw[fill=white] (1,0) circle (.35);
\draw[fill=white] (-1,0) circle (.35);
\draw[fill=white] (0,1) circle (.35);
\draw[fill=white] (0,-1) circle (.35);
\node at (0,0) {$z_i$};
\node at (-1,0) {$+$};
\node at (1,0) {$-$};
\node at (0,1) [scale=.7]{$n\!\!+\!\!1$};
\node at (0,-1) {$n$};
\end{tikzpicture} 
\\
\hline
\text{$P$-weights}&
1 & z_i & z_i(1-t^{n+1}) & 1\\
\hline
\text{$R$-weights}&
1 & z_i & z_i(1-t) & \frac{1-t^{n+1}}{1-t}\\
\hline\end{array}\]
\caption{Uncolored Boltzmann weights of two types: the $P$-weights (which coincide with those in \cite{KorffVerlinde}, and the $R$-weights.}
\label{fig:uncolored_weights}
\end{figure}

\newcommand{\ybelhs}[4]{ \begin{tikzpicture}[baseline=(current bounding box.center)]
  \draw (0,1) to [out = 0, in = 180] (2,3) to (4,3);
  \draw (0,3) to [out = 0, in = 180] (2,1) to (4,1);
  \draw (3,0) to (3,4);
  \draw[fill=white] (0,1) circle (.3);
  \draw[fill=white] (0,3) circle (.3);
  \draw[fill=white] (3,4) circle (.3);
  \draw[fill=white] (4,3) circle (.3);
  \draw[fill=white] (4,1) circle (.3);
  \draw[fill=white] (3,0) circle (.3);
  \node at (0,1) {$a$};
  \node at (0,3) {$b$};
  \node at (3,4) {$c$};
  \node at (4,3) {$d$};
  \node at (4,1) {$e$};
  \node at (3,0) {$f$};
\path[fill=white] (1,2) circle (.3);
\node at (1,2) {$#1$};
\path[fill=white] (3,3) circle (#4);
\node at (3,3) {$#2$};
\path[fill=white] (3,1) circle (#4);
\node at (3,1) {$#3$};
\end{tikzpicture}}

\newcommand{\yberhs}[4]{
\begin{tikzpicture}[baseline=(current bounding box.center)]
  \draw (0,1) to (2,1) to [out = 0, in = 180] (4,3);
  \draw (0,3) to (2,3) to [out = 0, in = 180] (4,1);
  \draw (1,0) to (1,4);
  \draw[fill=white] (0,1) circle (.3);
  \draw[fill=white] (0,3) circle (.3);
  \draw[fill=white] (1,4) circle (.3);
  \draw[fill=white] (4,3) circle (.3);
  \draw[fill=white] (4,1) circle (.3);
  \draw[fill=white] (1,0) circle (.3);
  \node at (0,1) {$a$};
  \node at (0,3) {$b$};
  \node at (1,4) {$c$};
  \node at (4,3) {$d$};
  \node at (4,1) {$e$};
  \node at (1,0) {$f$};
\path[fill=white] (3,2) circle (.4);
\node at (3,2) {$#1$};
\path[fill=white] (1,1) circle (#4);
\node at (1,1) {$#2$};
\path[fill=white] (1,3) circle (#4);
\node at (1,3) {$#3$};
\end{tikzpicture}}

%% Figure 3
\begin{figure}[h]
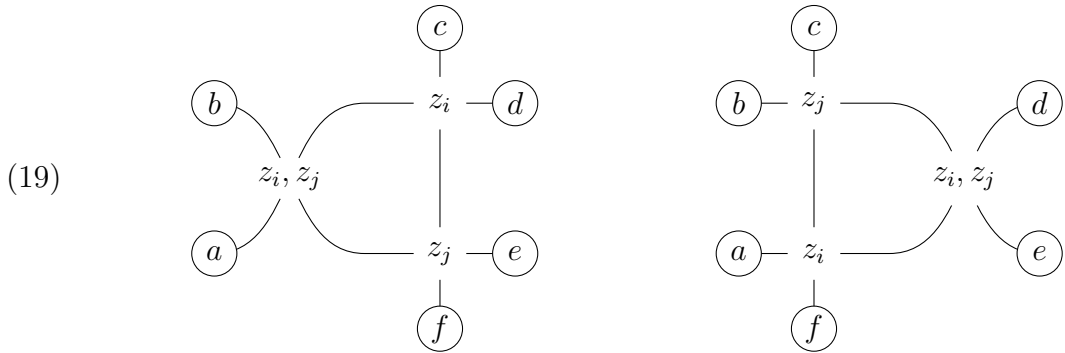

\begin{equation}
  \label{ybe}
  \begin{array}{ccc}\ybelhs{z_i,z_j}{z_i}{z_j}{0.35}&\qquad\qquad&
  \yberhs{z_i,z_j}{z_i}{z_j}{0.35}\end{array}
\end{equation}
\caption{The Yang-Baxter equation. 
The partition functions of the two small 3-vertex systems are the same. 
Here $a,b,c,d,e,f$ are the fixed boundary spins, and in each case
we sum over the possible spins of the three interior edges.
We may use either the $P$- or the $R$-weights.
\textbf{Uncolored case:} Use the weights
from Figure~\ref{fig:uncolored_weights} and the R-matrix from Figure~\ref{fig:uncolored_rmatrix}. 
\textbf{Colored case:} Use the colored weights obtained by fusion (Figure~\ref{fig:fusion}) from the monochrome weights (Figure~\ref{fig:monochrome}).}
\label{fig:ybe}
\end{figure}

%% Figure 4
% \input{figures/uncolored_rmatrix}%
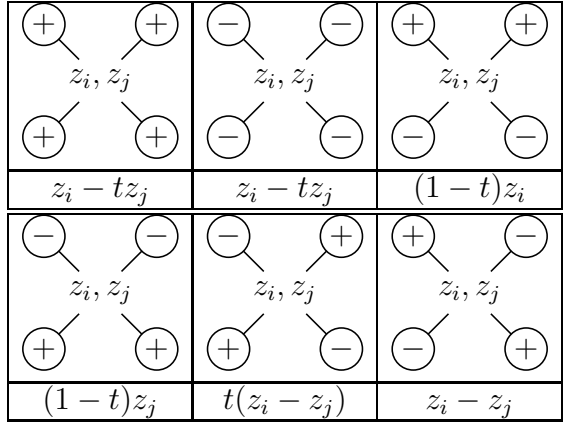
\begin{figure}[h]
  \[\begin{array}{|c|c|c|}
  \hline
  \begin{tikzpicture}
    \draw[semithick] (-.75,-.75)--(.75,.75);
    \draw[semithick] (-.75,.75)--(.75,-.75);
    \draw[semithick,fill=white] (-.75,-.75) circle (.28);
    \draw[semithick,fill=white] (-.75,.75) circle (.28);
    \draw[semithick,fill=white] (.75,.75) circle (.28);
    \draw[semithick,fill=white] (.75,.-.75) circle (.28);
    \node at (-.75,-.75) {$+$};
    \node at (-.75,.75) {$+$};
    \node at (.75,.75) {$+$};
    \node at (.75,.-.75) {$+$};
    \path[semithick,fill=white] (0,0) circle (.4);
    \node at (0,0) {$z_i,z_j$};
  \end{tikzpicture} &
  \begin{tikzpicture}
    \draw[semithick] (-.75,-.75)--(.75,.75);
    \draw[semithick] (-.75,.75)--(.75,-.75);
    \draw[semithick,fill=white] (-.75,-.75) circle (.28);
    \draw[semithick,fill=white] (-.75,.75) circle (.28);
    \draw[semithick,fill=white] (.75,.75) circle (.28);
    \draw[semithick,fill=white] (.75,.-.75) circle (.28);
    \node at (-.75,-.75) {$-$};
    \node at (-.75,.75) {$-$};
    \node at (.75,.75) {$-$};
    \node at (.75,.-.75) {$-$};
    \path[fill=white] (0,0) circle (.4);
    \node at (0,0) {$z_i,z_j$};
  \end{tikzpicture} &
  \begin{tikzpicture}
    \draw[semithick] (-.75,-.75)--(.75,.75);
    \draw[semithick] (-.75,.75)--(.75,-.75);
    \draw[semithick,fill=white] (-.75,-.75) circle (.28);
    \draw[semithick,fill=white] (-.75,.75) circle (.28);
    \draw[semithick,fill=white] (.75,.75) circle (.28);
    \draw[semithick,fill=white] (.75,.-.75) circle (.28);
    \node at (-.75,-.75) {$-$};
    \node at (-.75,.75) {$+$};
    \node at (.75,.75) {$+$};
    \node at (.75,.-.75) {$-$};
    \path[fill=white] (0,0) circle (.4);
    \node at (0,0) {$z_i,z_j$};
  \end{tikzpicture} 
  \\
  \hline
  z_i-tz_j & z_i-tz_j & (1-t)z_i
  \\
  \hline
  \hline
  \begin{tikzpicture}
    \draw[semithick] (-.75,-.75)--(.75,.75);
    \draw[semithick] (-.75,.75)--(.75,-.75);
    \draw[semithick,fill=white] (-.75,-.75) circle (.28);
    \draw[semithick,fill=white] (-.75,.75) circle (.28);
    \draw[semithick,fill=white] (.75,.75) circle (.28);
    \draw[semithick,fill=white] (.75,.-.75) circle (.28);
    \node at (-.75,-.75) {$+$};
    \node at (-.75,.75) {$-$};
    \node at (.75,.75) {$-$};
    \node at (.75,.-.75) {$+$};
    \path[fill=white] (0,0) circle (.4);
    \node at (0,0) {$z_i,z_j$};
  \end{tikzpicture} &
  \begin{tikzpicture}
    \draw[semithick] (-.75,-.75)--(.75,.75);
    \draw[semithick] (-.75,.75)--(.75,-.75);
    \draw[semithick,fill=white] (-.75,-.75) circle (.28);
    \draw[semithick,fill=white] (-.75,.75) circle (.28);
    \draw[semithick,fill=white] (.75,.75) circle (.28);
    \draw[semithick,fill=white] (.75,.-.75) circle (.28);
    \node at (.75,-.75) {$-$};
    \node at (.75,.75) {$+$};
    \node at (-.75,.75) {$-$};
    \node at (-.75,-.75) {$+$};
    \path[fill=white] (0,0) circle (.4);
    \node at (0,0) {$z_i,z_j$};
  \end{tikzpicture} &
  \begin{tikzpicture}
    \draw[semithick] (-.75,-.75)--(.75,.75);
    \draw[semithick] (-.75,.75)--(.75,-.75);
    \draw[semithick,fill=white] (-.75,-.75) circle (.28);
    \draw[semithick,fill=white] (-.75,.75) circle (.28);
    \draw[semithick,fill=white] (.75,.75) circle (.28);
    \draw[semithick,fill=white] (.75,.-.75) circle (.28);
    \node at (-.75,-.75) {$-$};
    \node at (-.75,.75) {$+$};
    \node at (.75,.75) {$-$};
    \node at (.75,-.75) {$+$};
    \path[fill=white] (0,0) circle (.4);
    \node at (0,0) {$z_i,z_j$};
  \end{tikzpicture} 
  \\
  \hline
  (1-t) z_j & t(z_i-z_j) & z_i-z_j\\
  \hline
  \end{array}\]
  \caption{The uncolored R-matrix. This works for either the uncolored $R$- or $P$-models.} 
  \label{fig:uncolored_rmatrix}
\end{figure}

In this section we introduce the uncolored models,
based on a single boson type $\minus$. The spinset of the horizontal edges is
$\Sigma_{\h}^{\unc} = \{ \plus, \minus \}$, where $\plus$ denotes the absence
of a particle, and $\minus$ denotes the presence. The spinset of the vertical edge is 
\[\Sigma_{\v}^{\unc}=\{\boson{n}|n\in\mathbb{N}\},\qquad
\mathbb{N}= \{ 0, 1, 2, \cdots \},\] 
where $\boson{n}$ represents a state with $n$ identical particles present. 
We use two types of Boltzmann weights, namely $P$- and $R$- weights. 
Both types depend on a parameter $z = z_i \in \mathbb{C}^\times$ and the weights of admissible vertices are described in Figure~\ref{fig:uncolored_weights}. The Boltzmann weights of all other vertices are zero. 

We note that the
$P$-weights are the same is those in Korff~{\cite{KorffVerlinde}} and $R$-weights are the special case of weights in Borodin and Wheeler~\cite{BorodinWheelerColored}, when $n=1$ and $s = 0$. Despite being the special case, $R$-weights have special aspects that warrant giving it a separate treatment. 

We will (for the remainder of the paper) take $\Lambda$ to be
the $\GL_r$ root system, which can be identified with $\mathbb{Z}^r$.
The Weyl group $W$ is then the symmetric group $S_r$. A 
weight $\lambda\in\Lambda$ is dominant if $\lambda_1\geqslant \cdots \geqslant\lambda_r$.
If $\lambda_r\geqslant  0$, then $\lambda$ is just a partition of
length $\leqslant r$. We will at first describe the systems when
$\lambda$ is a partition, then in Remark~\ref{rem:dominantpf} 
discuss the minor modifications required when $\lambda_r$ is allowed 
to be negative.

Let $\lambda = (\lambda_1, \cdots, \lambda_r)$ be a partition. Then we will
describe two uncolored systems $\mathfrak{S}^P_{\lambda} (\mathbf{z};t)$ and
$\mathfrak{S}^R_{\lambda} (\mathbf{z};t)$ where $\mathbf{z}= (z_1, \cdots,
z_r) \in \hat{T} (\mathbb{C}) =\mathbb{C}^r$, which we call the \emph{spectral parameters} for the corresponding model. Each uncolored system has $r$ rows, and $N$ columns, where $N \geqslant \lambda_1$. The rows are labeled
from $1$ to $r$ from top to bottom, and the columns are labeled from $0$ to
$N$ from right to left. The boundary
edges are assigned spins as follows. On the left edge, the horizontal edge of
every row is assigned spin $\plus$, and on the right edge every horizontal edge is
assigned spin $\minus$. On the bottom, the vertical edge of every column is
assigned spin $\boson{0}$, and on the top, the edge in the column labeled $j$ is
assigned spin $\boson{m_j}$ where $m_j$ is the number of parts $\lambda_i$ of
$\lambda$ equal to $j$. See Figure~\ref{fig:uncolored_model} for an
example of these boundary conditions.

A lattice model is called \emph{integrable} or \emph{exactly solvable} if the Boltzmann weights of its vertices satisfy the Yang-Baxter equation, which is a local equation for the weights that gives the relation between the Boltzmann weights of the vertices of different types. To describe the Yang-Baxter equation, we use the Boltzmann weights from
Figure~\ref{fig:uncolored_rmatrix}, which will work both for the $P$- and the
$R$-weights. 

\begin{proposition}[Kulish~\cite{KulishNonlinear}]
The following Yang-Baxter equation is satisfied: given
spins $a,b,d,e\in\Sigma_{\text{h}}$ and $c,f\in\Sigma_{\text{v}}$
the partition functions of the small systems in Figure~\ref{fig:ybe}
are equal.
\end{proposition}

\begin{proof}
This may be checked by examination of the individual cases.
\end{proof}

We will use the notations $\beta_R$ and $\beta_P$ to distinguish 
the two cases. We will use the notations $Z_P$ and $Z_R$ to denote the partition
functions of the various systems with respect to the $P$- and $R$-weights. Later we will also write $\beta_R^{\unc}$ or $\beta_R^{\col}$ to 
distinguish the weights for the uncolored and colored models. 

\begin{proposition}The partition functions 
$Z_R(\mathfrak{S}_\lambda(\mathbf{z};t))$ and
$Z_P(\mathfrak{S}_\lambda(\mathbf{z};t))$ are symmetric polynomials in 
the~$z_i$.
\end{proposition}

\begin{proof}
We will denote by $Z(\mathfrak{S}_\lambda(\mathbf{z};t))$ either
$Z_P(\mathfrak{S}_\lambda(\mathbf{z},t))$ or $Z_R(\mathfrak{S}_\lambda(\mathbf{z};t))$,
since both satisfy the Yang-Baxter equation for the same R-matrix, and
the proof is identical for the two cases.
We use the standard ``train argument,'' due to Baxter, and very
similar to the argument in~\cite{hkice}. Let $s_i$ be the simple reflection
that acts on $\mathbf{z}=(z_1,\cdots,z_r)$ by interchanging $z_i$ and $z_{i+1}$.
We attach the R-matrix to the grid between the $i$ and $i+1$ rows, with
boundary conditions as follows:
%% train1
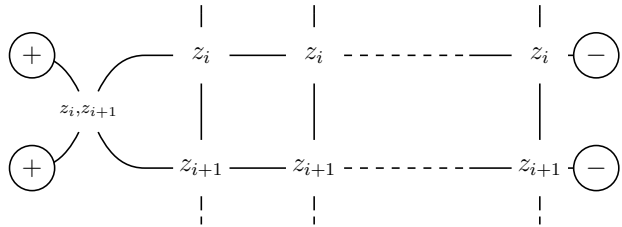
\[\begin{tikzpicture}[scale=.75,every node/.style={scale=0.85}]
\foreach \j in {1,3,7}
{
\draw[semithick,dashed] (\j,0)--(\j,.25);
\draw[semithick] (\j,.25)--(\j,3.75);
\draw[semithick,dashed] (\j,3.75)--(\j,4);
}
\draw[semithick] (-2,1) to [out=0,in=180] (0,3) to (3.25,3);
\draw[semithick] (-2,3) to [out=0,in=180] (0,1) to (3.25,1);
\foreach \i in {1,3}
{
\draw[semithick,dashed] (3.25,\i)--(5.75,\i);
\draw[semithick] (5.75,\i)--(8,\i);
}
\foreach \i/\j/\k in {1/1/i+1,3/1/i+1,7/1/i+1,1/3/i,3/3/i,7/3/i}
{
\path[fill=white] (\i,\j) circle (.5);
\node at (\i,\j) {$z_{\k}$};
}
\foreach \i/\j/\k in {-2/1/+,-2/3/+,8/1/-,8/3/-}
{
\draw[semithick,fill=white] (\i,\j) circle (.4);
\node at (\i,\j) {$\k$};
}
\path[fill=white] (-1,2) circle (.4);
\node at (-1,2) {$\scriptstyle z_i,z_{i+1}$};
\end{tikzpicture}\]
Note that there is only one configuration for the R-matrix with nonzero
Boltzmann weight, in which all adjacent edges have spin $\plus$. The
Boltzmann weight of this vertex is $z_i-tz_{i+1}$ Therefore
the Boltzmann weight of this configuration is 
\begin{equation}
\label{eq:leftattach}
(z_i-tz_{i+1})Z(\mathfrak{S}_\lambda(\mathbf{z};t))
\end{equation}
Now after using the Yang-Baxter equation $N$ times, we arrive at the
configuration:
%% train2
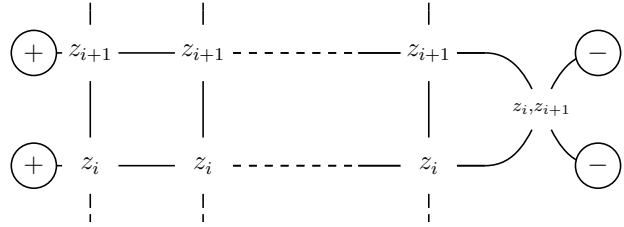
\[\begin{tikzpicture}[scale=.75,every node/.style={scale=0.85}]
\foreach \j in {1,3,7}
{
\draw[semithick,dashed] (\j,0)--(\j,.25);
\draw[semithick] (\j,.25)--(\j,3.75);
\draw[semithick,dashed] (\j,3.75)--(\j,4);
}
\foreach \i in {1,3}
{
	\draw[semithick] (0,\i)--(3.25,\i);
	\draw[semithick,dashed] (3.25,\i)--(5.75,\i);
	\draw[semithick] (5.75,\i) to (8,\i);
}
\draw[semithick] (8,3) to [out=0,in=180] (10,1);
\draw[semithick] (8,1) to [out=0,in=180] (10,3);
\foreach \i in {1,3}
{
\draw[semithick,dashed] (3.25,\i)--(5.75,\i);
\draw[semithick] (5.75,\i)--(8,\i);
}
\foreach \i/\j/\k in {1/1/i,3/1/i,7/1/i,1/3/i+1,3/3/i+1,7/3/i+1}
{
\path[fill=white] (\i,\j) circle (.5);
\node at (\i,\j) {$z_{\k}$};
}
\foreach \i/\j/\k in {0/1/+,0/3/+,10/1/-,10/3/-}
{
\draw[semithick,fill=white] (\i,\j) circle (.4);
\node at (\i,\j) {$\k$};
}
\path[fill=white] (9,2) circle (.4);
\node at (9,2) {$\scriptstyle z_i,z_{i+1}$};
\end{tikzpicture}\]
Again there is only one possible configuration for the R-matrix,
in which all adjacent edges have spin $\minus$, and the Boltzmann weight
of this vertex is again $z_i-tz_{i+1}$. Therefore (\ref{eq:leftattach})
equals
\[(z_i-tz_{i+1})Z(\mathfrak{S}_\lambda(s_i\mathbf{z};t))\]
since the $z_i$ and $z_{i+1}$ rows have been switched. Comparing, we
see that $Z(\mathfrak{S}_\lambda(\mathbf{z};t))$ is invariant under
$s_i$, hence under all permutations of the~$\mathbf{z}$.
\end{proof}

In the following result, we show that the partition functions with $P$- and $R$- weights are related to each other by a simple transformation. In Theorem~\ref{thm:uncoloredpf}, we will show that the partition functions with $P$- and $R$-weights are equal to the Hall-Littlewood polynomials $P_\lambda$ and $R_\lambda$, respectively. 
If $m$ is a nonnegative integer, we will use the notation
\[v_m(t)=\prod_{i=1}^m\frac{1-t^i}{1-t}\]
and if $\lambda$ is a partition, we will denote
\begin{equation}
\label{eq:vlamdef}
v_\lambda(t)=\prod_{i\geqslant 0}v_{m_i}(t),
\end{equation}
where $m_i$ is the number of parts of $\lambda$ equal to $i$. These notations
are introduced in~\cite{MacdonaldBook}, Section~3.1 in connection with
the definition of the Hall-Littlewood polynomials.

\begin{proposition}
  \label{prop:uncoloredpf}
  The partition functions with $P$- and $R$- weights differ only by a factor which is independent on $\mathbf{z}$:
  \[\frac{Z_R(\mathfrak{S}_\lambda(\mathbf{z};t))}{Z_P(\mathfrak{S}_\lambda(\mathbf{z};t))}=v_\lambda(t).\]
\end{proposition}

\begin{proof}
This will be proved by comparing the \textit{column transfer matrices},
which we define as follows. Let $0\leqslant j \leqslant N$, and let
$m=m_j$ be the number of entries in $\lambda$ that are equal
to $j$. Let
\[\delta_1,\cdots,\delta_r,
\varepsilon_1,\cdots,\varepsilon_r\in\Sigma_{\h}=\{\plus,\minus\}.\]
Now let us consider the contributions $C^P_m(\delta,\varepsilon)$
$C^R_m(\delta,\varepsilon)$ of the vertices in the $j$-th column
of a state $\mathfrak{s}$ of the model having spins
$\delta_i$ to the left of the vertices and $\varepsilon_i$ to
the right. (With $m$ fixed we may regard $C^P_m$ or $C^R_m$ as
a matrix with indices $\delta,\varepsilon$, and these are the column transfer
matrices.)
These values are zero unless the number of $\varepsilon_i$
that are equal to $\minus$ exceeds the number of $\delta_i$ that
are equal to $\minus$ by exactly $m$. From Figure~\ref{fig:uncolored_weights}
it is easy to see that if this is true, then
\[\frac{C^R_m(\delta,\varepsilon)}{C^P_m(\delta,\varepsilon)}=
\prod_{i=1}^m\frac{1-t^i}{1-t}.\]
Multiplying this identity over all $j$ gives
\[\frac{Z_R(\mathfrak{S}_\lambda(\mathbf{z};t))}{Z_P(\mathfrak{S}_\lambda(\mathbf{z};t))}=\prod_j v_{m_j}(t).\qedhere\]
\end{proof}

The uncolored models satisfy the Yang-Baxter equations, which imply that the partition functions are symmetric in 
the $\mathbf{z}$. These partition functions are polynomials in the variables $\mathbf{z}$ and $t$, homogeneous
of degree $k$ in $z_1,\cdots,z_r$ when $\lambda$ is a partition of $k$. However, this information alone is insufficient 
to evaluate the partition functions. Additional information is needed to compute them, such as the degree of the 
partition function as a polynomial in $t$ and divisibility facts, which were used in a previous study~\cite{hkice}. 
Inductive reasoning, as shown in~\cite{KorffVerlinde} and Izergin's proof~\cite{IzerginDeterminant} of the 
Izergin-Korepin determinant formula, could be used to evaluate the partition functions. However, an extra 
ingredient is needed regardless of the method used. In this study, we present a different approach that works 
for both fermionic and bosonic colored models. We use a local lifting property of the Boltzmann weights to relate 
the uncolored model to the colored model and obtain the evaluation of the partition function. Thus, we provide a 
proof of Theorem~\ref{thm:uncoloredpf} using the Yang-Baxter equation.

\begin{remark}
\label{rem:dominantpf}
We have assumed so far that $\lambda$ is a partition, and that the
grid has columns labeled from $0$ to $N$ where $N$ is any integer
$\geqslant\lambda_1$. However we will eventually want to consider
systems indexed by dominant weights 
$\lambda_1\geqslant\cdots\geqslant\lambda_r$ with $\lambda_r$
possibly negative. Let $N$ and $M$ be any integers such that
$N\geqslant\lambda_1\geqslant\lambda_r\geqslant M$. Let us
consider what happens if we increase $N$ or decrease $M$. 
Let $\mathfrak{S}_{\lambda,M,N}(\mathbf{z};t)$ be the
resulting ($P$- or $R$-) system with $N+M+1$ columns numbered from $M$ to $N$. There is a bijection
between the states of $\mathfrak{S}_{\lambda,M,N}(\mathbf{z};t)$
and $\mathfrak{S}_{\lambda,M,N+1}(\mathbf{z};t)$;
adding a column to the left only adds $A(0)$ in
the left added column and does not change the partition
function. On the other hand, adding a column to the
right adds a $B(0)$ pattern in each row, and multiplies
the partition function by $z_1\cdots z_r$. Therefore
we may define 
\begin{equation}
\label{eq:dominantpf}
Z(\mathfrak{S}_{\lambda,M,N}(\mathbf{z};t))=
(z_1\cdots z_r)^MZ(\mathfrak{S}_{\lambda,M,N}(\mathbf{z};t))
\end{equation}
and this definition is independent of $M$.
\end{remark}

\section{The colored bosonic models}

%% Figure 8
% \input{figures/monochrome}%
\begin{figure}[tb]
\[\begin{array}{|c|c|c|c|c|}
\hline
\text{name} & A(n) & B(n) & C(n) & D(n) \\
\hline
& \begin{tikzpicture}[scale=1.2,every node/.style={scale=0.85}]
\draw (-1,0) to (1,0);
\draw (0,-1) to (0,1);
\path[fill=white] (0,0) circle (.32);
\draw[fill=white] (1,0) circle (.30);
\draw[fill=white] (-1,0) circle (.30);
\draw[fill=white] (0,1) circle (.30);
\draw[fill=white] (0,-1) circle (.30);
\node at (0,0) {$z_i,c$};
\node at (-1,0) {$+$};
\node at (1,0) {$+$};
\node at (0,-1) {$n$};
\node at (0,1) {$n$};
\node at (0,1.35) {$\quad$}; % to increase the bounding box
\end{tikzpicture} &
\begin{tikzpicture}[scale=1.2,every node/.style={scale=0.85}]
\draw (-1,0) to (1,0);
\draw (0,-1) to (0,1);
\path[fill=white] (0,0) circle (.32);
\draw[fill=white] (1,0) circle (.30);
\draw[fill=white] (-1,0) circle (.30);
\draw[fill=white] (0,1) circle (.30);
\draw[fill=white] (0,-1) circle (.30);
\node at (0,0) {$z_i,c$};
\node at (-1,0) {$d$};
\node at (1,0) {$d$};
\node at (0,-1) {$n$};
\node at (0,1) {$n$};
\node at (0,1.35) {$\quad$}; % to increase the bounding box
\end{tikzpicture} &
\begin{tikzpicture}[scale=1.2,every node/.style={scale=0.85}]
\draw (-1,0) to (1,0);
\draw (0,-1) to (0,1);
\path[fill=white] (0,0) circle (.32);
\draw[fill=white] (1,0) circle (.30);
\draw[fill=white] (-1,0) circle (.30);
\draw[fill=white] (0,1) circle (.30);
\draw[fill=white] (0,-1) circle (.30);
\node at (0,0) {$z_i,c$};
\node at (-1,0) {$c$};
\node at (1,0) {$+$};
\node at (0,1) {$n$};
\node at (0,-1) [scale=.75]{$n\!\!+\!\!1$};
\node at (0,1.35) {$\quad$}; % to increase the bounding box
\end{tikzpicture} &
\begin{tikzpicture}[scale=1.2,every node/.style={scale=0.85}]
\draw (-1,0) to (1,0);
\draw (0,-1) to (0,1);
\path[fill=white] (0,0) circle (.32);
\draw[fill=white] (1,0) circle (.30);
\draw[fill=white] (-1,0) circle (.30);
\draw[fill=white] (0,1) circle (.30);
\draw[fill=white] (0,-1) circle (.30);
\node at (0,0) {$z_i,c$};
\node at (-1,0) {$+$};
\node at (1,0) {$c$};
\node at (0,1) [scale=.7]{$n\!\!+\!\!1$};
\node at (0,-1) {$n$};
\node at (0,1.35) {$\quad$}; % to increase the bounding box
\end{tikzpicture} 
\\
\hline
\text{$P$-weights}&
1 & \begin{array}{cc}1 & c<d\\z_i&c=d\\ t^n&c>d\end{array} & z_i(1-t^{n+1}) & 1\\
\hline
\text{$R$-weights}&
1 & \begin{array}{cc}1 & c<d\\z_i&c=d\\ t^n&c>d\end{array} & z_i(1-t) & \frac{1-t^{n+1}}{1-t}\\
\hline\end{array}\]
\caption{Monochrome vertex type $z_i,c$. The vertical edge can carry only the
color $c$. The possible states of the horizontal edges are all colors
and $+$. Possible spins of the vertical edges are labeled by integers $n$,
representing $n$ copies of the boson of color~$c$.}
\label{fig:monochrome}
\end{figure}

%% Figure 9
% \input{figures/fusion}%
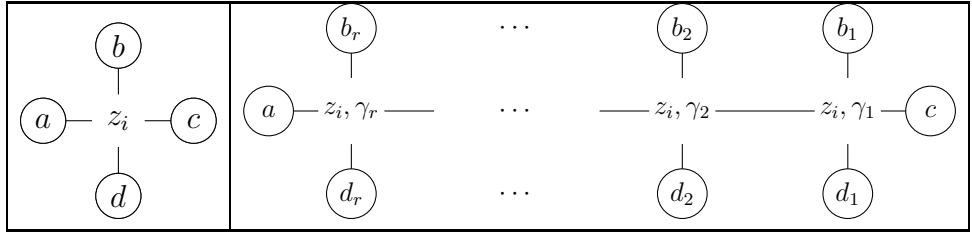
\begin{figure}[ht]
\[\begin{array}{|c|c|}\hline
\begin{tikzpicture}
\draw (0,1) to (2,1);
\draw (1,0) to (1,2);
\path[fill=white] (1,1) circle (.35);
\draw[fill=white] (0,1) circle (.3);
\draw[fill=white] (2,1) circle (.3);
\draw[fill=white] (1,0) circle (.3);
\draw[fill=white] (1,2) circle (.3);
\draw[fill=white] (0,1) circle (.3);
\draw[fill=white] (2,1) circle (.3);
\draw[fill=white] (1,0) circle (.3);
\draw[fill=white] (1,2) circle (.3);
\node at (1,1) {$z_i$};
\node at (0,1) {$a$};
\node at (2,1) {$c$};
\node at (1,2) {$b$};
\node at (1,0) {$d$};
\end{tikzpicture}&
\begin{tikzpicture}[scale=1.1,every node/.style={scale=0.85}]
\draw (0,1)--(2,1);
\draw (4,1)--(8,1);
\draw (1,0)--(1,2);
\draw (5,0)--(5,2);
\draw (7,0)--(7,2);
\node at (3,0) {$\cdots$};
\node at (3,1) {$\cdots$};
\node at (3,2) {$\cdots$};
\draw[fill=white] (0,1) circle (.3);
\draw[fill=white] (1,0) circle (.3);
\draw[fill=white] (1,2) circle (.3);
\draw[fill=white] (5,0) circle (.3);
\draw[fill=white] (5,2) circle (.3);
\draw[fill=white] (7,0) circle (.3);
\draw[fill=white] (7,2) circle (.3);
\draw[fill=white] (8,1) circle (.3);
\path[fill=white] (1,1) circle (.4);
\path[fill=white] (5,1) circle (.4);
\path[fill=white] (7,1) circle (.4);
\node at (0,1) {$a$};
\node at (8,1) {$c$};
\node at (1,2) {$b_r$};
\node at (1,0) {$d_r$};
\node at (5,2) {$b_2$};
\node at (5,0) {$d_2$};
\node at (7,2) {$b_1$};
\node at (7,0) {$d_1$};
\node at (1,1) {$z_i,\gamma_r$};
\node at (5,1) {$z_i,\gamma_2$};
\node at (7,1) {$z_i,\gamma_1$};
\end{tikzpicture}
\\\hline\end{array}\]
\caption{Fusion. 
The colors $\gamma_i$ are ordered so that 
$\gamma_r < \gamma_{r-1} < \cdots < \gamma_1$, so they are arranged 
in increasing order from left to right.
This procedure replaces a sequence of vertices by a single vertex. 
Here $z_i,\gamma_i$ is the ``monochrome'' vertex from 
Figure~\ref{fig:monochrome}. 
The single fused vertex $z_i$ replaces the $r$ unfused vertices 
$z_i,\gamma_i$. If $b,d\in\Sigma^{\col}_{\v}$ then write
$b=\gamma_1^{b_1}\cdots\gamma_r^{b_r}\in\Sigma^{\col}_{\v}$ for
$(b_1,\cdots,b_r)\in\prod_{i=1}^r\Sigma^{\mon}_{\v,\gamma_i}$,
and similarly write
$d=\gamma_1^{d_1}\cdots\gamma_r^{d_r}\in\Sigma_{\v}^{\col}$.
The Boltzmann weight of the fused vertex (call it $v$)
is just the partition function of the configuration on the right-hand side
of the figure. That is, there will be a unique assigment of
spins to the internal edges on the right such that the
Boltzmann weights are nonzero, and $\beta(v)$ is the
product of the Boltzmann weights of the monochrome vertices,
from Figure~\ref{fig:monochrome}.}
\label{fig:fusion}
\end{figure}

%% Figure 5
% \input{figures/rmatrix}%
% \input{figures/ybeaux}%
\begin{figure}[b]
  \[\begin{array}{|c|c|c|c|}
  \hline
  \begin{tikzpicture}
    \draw[semithick] (-.75,-.75)--(.75,.75);
    \draw[semithick] (-.75,.75)--(.75,-.75);
    \draw[semithick,fill=white] (-.75,-.75) circle (.28);
    \draw[semithick,fill=white] (-.75,.75) circle (.28);
    \draw[semithick,fill=white] (.75,.75) circle (.28);
    \draw[fill=white] (.75,.-.75) circle (.28);
    \node at (-.75,-.75) {$+$};
    \node at (-.75,.75) {$+$};
    \node at (.75,.75) {$+$};
    \node at (.75,.-.75) {$+$};
    \path[semithick,fill=white] (0,0) circle (.4);
    \node at (0,0) {$z_i,z_j$};
  \end{tikzpicture} &
  \begin{tikzpicture}
    \draw[line width=0.4mm,red] (-.75,-.75)--(.75,.75);
    \draw[line width=0.4mm,red] (-.75,.75)--(.75,-.75);
    \draw[line width=0.4mm,red,fill=white] (-.75,-.75) circle (.28);
    \draw[line width=0.4mm,red,fill=white] (-.75,.75) circle (.28);
    \draw[line width=0.4mm,red,fill=white] (.75,.75) circle (.28);
    \draw[line width=0.4mm,red,fill=white] (.75,.-.75) circle (.28);
    \node at (-.75,-.75) {$c$};
    \node at (-.75,.75) {$c$};
    \node at (.75,.75) {$c$};
    \node at (.75,.-.75) {$c$};
    \path[fill=white] (0,0) circle (.4);
    \node at (0,0) {$z_i,z_j$};
  \end{tikzpicture} &
  \begin{tikzpicture}
    \draw[line width=0.4mm,blue] (-.75,-.75)--(0,0);
    \draw[line width=0.4mm,red] (-.75,.75)--(0,0);
    \draw[line width=0.4mm,red] (.75,.75)--(0,0);
    \draw[line width=0.4mm,blue] (.75,-.75)--(0,0);
    \draw[line width=0.4mm,blue,fill=white] (-.75,-.75) circle (.28);
    \draw[line width=0.4mm,red,fill=white] (-.75,.75) circle (.28);
    \draw[line width=0.4mm,red,fill=white] (.75,.75) circle (.28);
    \draw[line width=0.4mm,blue,fill=white] (.75,.-.75) circle (.28);
    \node at (-.75,-.75) {$c$};
    \node at (-.75,.75) {$d$};
    \node at (.75,.75) {$d$};
    \node at (.75,.-.75) {$c$};
    \path[fill=white] (0,0) circle (.4);
    \node at (0,0) {$z_i,z_j$};
  \end{tikzpicture} &
  \begin{tikzpicture}
    \draw[line width=0.4mm,blue] (-.75,-.75)--(0,0);
    \draw[line width=0.4mm,red] (-.75,.75)--(0,0);
    \draw[line width=0.4mm,blue] (.75,.75)--(0,0);
    \draw[line width=0.4mm,red] (.75,-.75)--(0,0);
    \draw[line width=0.4mm,blue,fill=white] (-.75,-.75) circle (.28);
    \draw[line width=0.4mm,red,fill=white] (-.75,.75) circle (.28);
    \draw[line width=0.4mm,blue,fill=white] (.75,.75) circle (.28);
    \draw[line width=0.4mm,red,fill=white] (.75,.-.75) circle (.28);
    \node at (-.75,-.75) {$c$};
    \node at (-.75,.75) {$d$};
    \node at (.75,.75) {$c$};
    \node at (.75,.-.75) {$d$};
    \path[fill=white] (0,0) circle (.4);
    \node at (0,0) {$z_i,z_j$};
  \end{tikzpicture} 
  \\
  \hline
  z_i-tz_j & z_i-tz_j &
  \begin{array}{ll} (1-t)z_i &\text{if $c<d$}\\(1-t)z_j &\text{if
      $c>d$}\end{array} & \begin{array}{ll} z_i-z_j &\text{if $c>d$}\\t(z_i-z_j) &\text{if $c<d$}\end{array} \\
  \hline
  \hline
  \begin{tikzpicture}
    \draw[line width=0.4mm,red] (-.75,-.75)--(0,0);
    \draw[semithick] (-.75,.75)--(0,0);
    \draw[semithick] (.75,.75)--(0,0);
    \draw[line width=0.4mm,red] (.75,-.75)--(0,0);
    \draw[line width=0.4mm,red,fill=white] (-.75,-.75) circle (.28);
    \draw[semithick,fill=white] (-.75,.75) circle (.28);
    \draw[semithick,fill=white] (.75,.75) circle (.28);
    \draw[line width=0.4mm,red,fill=white] (.75,.-.75) circle (.28);
    \node at (-.75,-.75) {$c$};
    \node at (-.75,.75) {$+$};
    \node at (.75,.75) {$+$};
    \node at (.75,.-.75) {$c$};
    \path[fill=white] (0,0) circle (.4);
    \node at (0,0) {$z_i,z_j$};
  \end{tikzpicture} &
  \begin{tikzpicture}
    \draw[semithick] (-.75,-.75)--(0,0);
    \draw[line width=0.4mm,red] (-.75,.75)--(0,0);
    \draw[line width=0.4mm,red] (.75,.75)--(0,0);
    \draw[semithick] (.75,.-.75)--(0,0);
    \draw[semithick,fill=white] (-.75,-.75) circle (.28);
    \draw[line width=0.4mm,red,fill=white] (-.75,.75) circle (.28);
    \draw[line width=0.4mm,red,fill=white] (.75,.75) circle (.28);
    \draw[semithick,fill=white] (.75,-.75) circle (.28);
    \node at (-.75,-.75) {$+$};
    \node at (-.75,.75) {$c$};
    \node at (.75,.75) {$c$};
    \node at (.75,.-.75) {$+$};
    \path[fill=white] (0,0) circle (.4);
    \node at (0,0) {$z_i,z_j$};
  \end{tikzpicture} &
  \begin{tikzpicture}
    \draw[semithick] (-.75,-.75)--(0,0);
    \draw[line width=0.4mm,red] (-.75,.75)--(0,0);
    \draw[semithick] (.75,.75)--(0,0);
    \draw[line width=0.4mm,red] (.75,-.75)--(0,0);
    \draw[semithick,fill=white] (-.75,-.75) circle (.28);
    \draw[line width=0.4mm,red,fill=white] (-.75,.75) circle (.28);
    \draw[semithick,fill=white] (.75,.75) circle (.28);
    \draw[line width=0.4mm,red,fill=white] (.75,-.75) circle (.28);
    \node at (-.75,-.75) {$+$};
    \node at (-.75,.75) {$c$};
    \node at (.75,.75) {$+$};
    \node at (.75,.-.75) {$c$};
    \path[fill=white] (0,0) circle (.4);
    \node at (0,0) {$z_i,z_j$};
  \end{tikzpicture} &
  \begin{tikzpicture}
    \draw[line width=0.4mm,red] (-.75,-.75)--(0,0);
    \draw[semithick] (-.75,.75)--(0,0);
    \draw[line width=0.4mm,red] (.75,.75)--(0,0);
    \draw[semithick] (.75,-.75)--(0,0);
    \draw[line width=0.4mm,red,fill=white] (-.75,-.75) circle (.28);
    \draw[semithick,fill=white] (-.75,.75) circle (.28);
    \draw[line width=0.4mm,red,fill=white] (.75,.75) circle (.28);
    \draw[semithick,fill=white] (.75,-.75) circle (.28);
    \node at (-.75,-.75) {$c$};
    \node at (-.75,.75) {$+$};
    \node at (.75,.75) {$c$};
    \node at (.75,-.75) {$+$};
    \path[fill=white] (0,0) circle (.4);
    \node at (0,0) {$z_i,z_j$};
  \end{tikzpicture} 
  \\
  \hline
  (1-t) z_i & (1-t) z_j &
  t(z_i-z_j) & z_i-z_j\\
  \hline
  \end{array}\]
  \caption{The R-matrix. This R-matrix is nearly identical to Figure~6 in \cite{BBBGIwahori}, 
  with one difference: for the first entry ($\small{+}{+}{+}{+}$), the R-matrix value is $z_i-tz_j$ here, 
  but $z_j-tz_i$ in~\cite{BBBGIwahori}. This small difference makes a difference in the
  quantum group: in~\cite{BBBGIwahori} the quantum group is a Drinfeld twist of 
  $U_{\sqrt{t}}(\widehat{\mathfrak{gl}}(r|1))$, but this R-matrix is that of
  $U_{\sqrt{t}}(\widehat{\mathfrak{sl}}(r+1))$. (Compare~\cite{JimboToda}, equation~(3.5).)}
  \label{fig:rmatrix}
\end{figure}

%% Figure 6
\begin{figure}[ht]
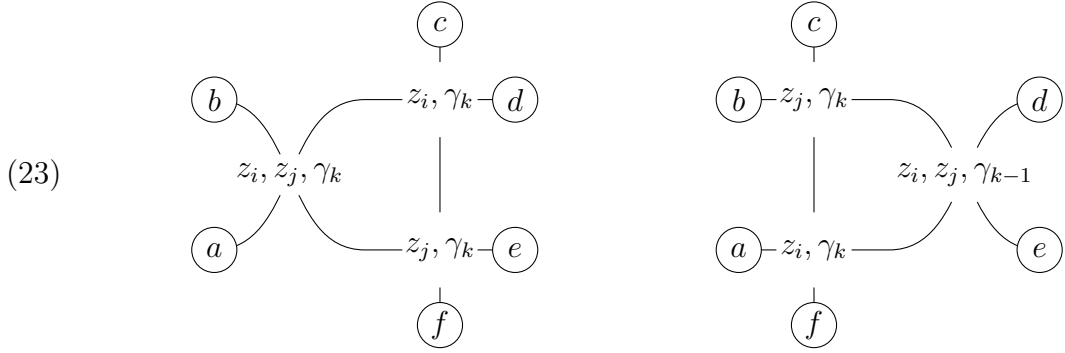

\begin{equation}
  \label{ybepre}
  \begin{array}{ccc}
  \ybelhs{z_i,z_j,\gamma_k}{z_i,\gamma_k}{z_j,\gamma_k}{0.5}&\qquad\qquad&
  \yberhs{z_i,z_j,\gamma_{k-1}}{z_i,\gamma_k}{z_j,\gamma_k}{0.5}\end{array}
\end{equation}
\caption{Auxiliary Yang-Baxter equations. The auxiliary R-matrix labeled $z_i,z_j,\gamma_k$ is from Figure~\ref{fig:genrmatrix}. Our convention is that $\gamma_0=\gamma_r$, which agrees with the R-matrix in Figure~\ref{fig:rmatrix} (that is, if there is only one column, then we have $\gamma_0 = \gamma_1$, and $\gamma$ can be omitted.). Note that
we are using the monochrome vertices so $c,f\in\Sigma^{\mon}_{\v,\gamma_k}$.
This Yang-Baxter equation may be proved by direct examination of the possible cases.}
\label{fig:ybeaux}
\end{figure}

\newcommand{\gamgamunco}[5]{
    \draw (-.75,-.75)--(.75,.75);
    \draw (-.75,.75)--(.75,-.75);
    \draw[fill=white] (-.75,-.75) circle (.28);
    \draw[fill=white] (-.75,.75) circle (.28);
    \draw[fill=white] (.75,.75) circle (.28);
    \draw[fill=white] (.75,.-.75) circle (.28);
    \node at (-.75,-.75) {$#1$};
    \node at (-.75,.75) {$#2$};
    \node at (.75,.75) {$#3$};
    \node at (.75,.-.75) {$#4$};
    \path[fill=white] (0,0) circle (.4);
    \node at (0,0) {$#5$};
    \node at (0,1) { };}

\newcommand{\spokea}[2]{
  \draw[line width=0.5mm, #1] (-.75,-.75)--(-.25,-.25);
  \draw[line width=0.5mm, #1,fill=white] (-.75,-.75) circle (.28);
  \node at (-.75,-.75) {$#2$};
% \node at (-.75,-.75-.5) {$#2$};
}

\newcommand{\spokeb}[2]{
  \draw[line width=0.5mm, #1] (-.75,.75)--(-.25,.25);
  \draw[line width=0.5mm, #1,fill=white] (-.75,.75) circle (.28);
  \node at (-.75,.75) {$#2$};
% \node at (-.75,.75+.5) {$#2$};
}

\newcommand{\spokec}[2]{
  \draw[line width=0.5mm, #1] (.75,.75)--(.25,.25);
  \draw[line width=0.5mm, #1,fill=white] (.75,.75) circle (.28);
  \node at (.75,.75) {$#2$};
% \node at (.75,.75+.5) {$#2$};
}

\newcommand{\spoked}[2]{
  \draw[line width=0.5mm, #1] (.75,-.75)--(.25,-.25);
  \draw[line width=0.5mm, #1,fill=white] (.75,.-.75) circle (.28);
  \node at (.75,-.75) {$#2$};
% \node at (.75,-.75-.5) {$#2$};
}

\newcommand{\spokepa}{
% \node at (-.75,-.75-.5) { };
  \node at (-.75,-.75) {$+$};
}

\newcommand{\spokepb}{
% \node at (-.75,.75+.5) { };
  \node at (-.75,.75) {$+$};
}

\newcommand{\spokepc}{
% \node at (.75,.75+.5) { };
  \node at (.75,.75) {$+$};
}

\newcommand{\spokepd}{
% \node at (.75,-.75-.5) { };
  \node at (.75,-.75) {$+$};
}

\newcommand{\cv}[2]{
%  \draw[line width=0.5mm, #1, fill=white] (0,0) circle (.25);
  \path[fill=white] (0,0) circle (.25);
  \node at (0,0) {$z_i,z_j,c$};
}

% Figure 7
\begin{figure}[htb]
  \[
  \begin{array}{|c|c|c|c|}
    \hline
    \begin{tikzpicture}\gamgamunco{+}{+}{+}{+}{c}\spokepa\spokepb\spokepc\spokepd\cv{red}{c}\end{tikzpicture} &
    \begin{tikzpicture}\gamgamunco{d}{d}{d}{d}{c}\spokea{blue}{d}\spokeb{blue}{d}\spokec{blue}{d}\spoked{blue}{d}\cv{red}{c}\end{tikzpicture} &
    \begin{tikzpicture}\gamgamunco{d}{e}{d}{e}{c}\cv{red}{c}
      \spokea{blue}{d}\spokeb{darkgreen}{e}\spokec{blue}{d}\spoked{darkgreen}{e}\cv{red}{c}\end{tikzpicture} &
    \begin{tikzpicture}\gamgamunco{d}{e}{e}{d}{c}\cv{red}{c}
      \spokea{blue}{d}\spokeb{darkgreen}{e}\spokec{darkgreen}{e}\spoked{blue}{d}\cv{red}{c}\end{tikzpicture} \\
    \hline
    z_i-tz_j & \begin{array}{c}z_i-tz_j\\\text{$c=d$ allowed.}\end{array} &
      \begin{array}{c}\begin{array}{ll} t(z_i-z_j)&\text{if $e>d$}\\z_i-z_j&\text{if $d>e$}\\\end{array}\\\text{$c=d$ or $e$ allowed}\end{array}&
      \begin{array}{ll}
        (1-t)z_j&\text{if $e>c>d$}\\
                &\text{or $c>d>e$}\\
                &\text{or $d>e>c$}\\
        (1-t)z_i&\text{if $d>c>e$}\\
                &\text{or $c>e>d$}\\
                &\text{or $e>d>c$}\end{array}\\
    \hline
    \begin{tikzpicture}\gamgamunco{d}{c}{c}{d}{c}\cv{red}{c}\spokea{blue}{d}\spokeb{red}{c}\spokec{red}{c}\spoked{blue}{d}\cv{red}{c}\end{tikzpicture} &
    \begin{tikzpicture}\gamgamunco{c}{d}{d}{c}{c}\cv{red}{c}\spokea{red}{c}\spokeb{blue}{d}\spokec{blue}{d}\spoked{red}{c}\cv{red}{c}\end{tikzpicture} &
    \begin{tikzpicture}\gamgamunco{+}{d}{+}{d}{c}\spokepa\spokeb{blue}{d}\spokepc\spoked{blue}{d}\cv{red}{c}\end{tikzpicture}&
    \begin{tikzpicture}\gamgamunco{d}{+}{d}{+}{c}\spokea{blue}{d}\spokepb\spokec{blue}{d}\spokepd\cv{red}{c}\end{tikzpicture}\\
    \hline
    (1-t)z_j&(1-t)z_i&\begin{array}{c}t(z_i-z_j)\\\text{$c=d$ allowed}\end{array}&\begin{array}{c}z_i-z_j\\\text{$c=d$ allowed}\end{array}\\
    \hline
    &\begin{tikzpicture}\gamgamunco{d}{+}{+}{d}{c}\spokea{blue}{d}\spokepb\spokepc\spoked{blue}{d}\cv{red}{c}\end{tikzpicture}&
    \begin{tikzpicture}\gamgamunco{+}{d}{d}{+}{c}\spokepa\spokeb{blue}{d}\spokec{blue}{d}\spokepd\cv{red}{c}\end{tikzpicture}&\\
    \hline
    &\begin{array}{c}(1-t)z_i\\\text{$c=d$ allowed}\end{array}&
    \begin{array}{c}(1-t)z_j\\\text{$c=d$ allowed}\end{array}&\\
    \hline
  \end{array}\]
\caption{R-vertices for auxiliary Yang-Baxter equations. These are labeled by the spectral parameters $z_i$, $z_j$ and a color $c$.  Except for the first entry, these weights are identical to the weights in~\cite{BBBGIwahori}, Figure~11, where $z_i-tz_j$ in this paper is replaced by $z_j-tz_i$. This seemingly minor difference changes the quantum group from $U_q(\widehat{\mathfrak{gl}}(1|r))$ in~\cite{BBBGIwahori} to $U_q(\widehat{\mathfrak{gl}}(r+1))$ in this paper.}
  \label{fig:genrmatrix}
\end{figure}

Like the uncolored models, there are two variants of the colored models,
called the $R$-models and the $P$-models. Both kinds of models have two
different but equivalent realizations called \textit{fused} and \textit{unfused}.
The two versions have the same partition function, and both are
based on grids with $r$ rows, but the unfused model has $r$ times
as many columns, and hence $r$ times as many vertices and edges
(both horizontal and vertical) as the fused model. Unless otherwise
made plain, we will be describing the fused model in our explanations.
We will use the unfused model in a few crucial places, to define the 
Boltzmann weights, and to prove the Yang-Baxter equations and the
local lifting property.

Let us describe the spinsets of the colored models. There are separate
spinsets for the horizontal and vertical edges. The horizontal spinset is
\[ \Sigma_{\h}^{\col} = \{ +, \gamma_1, \cdots, \gamma_r \}, \]
where $+$ is a special \textit{vacuum color} and the 
elements $\gamma_i$ are called colors.
We give them an ordering so that $\gamma_1 > \gamma_2 > \cdots > \gamma_r$. The vertical
spinset $\Sigma^{\col}_{\v}$ consists of all multisets on $\{ \gamma_1,
\cdots, \gamma_r \}$. Thus an element $b$ of $\Sigma^{\col}_{\v}$
is a map $m : \{ \gamma_1, \cdots, \gamma_r \} \longrightarrow \mathbb{N}= \{ 0, 1, 2,
\cdots \}$, where $m (\gamma_i)$ is interpreted as the multiplicity of $\gamma_i$ in
$b$. If $m_i = m (\gamma_i)$ we may write $b = \gamma_1^{m_1} \cdots
\gamma_r^{m_r}$. We may denote the ``vacuum'' state where all $m_i=0$
as $\boson{0}$.

Let $c\in\{\gamma_1,\cdots,\gamma_r\}$ be a color.
In a state $\mathfrak{s}$ of the model, we recall that every edge $E$
is assigned a spin $\mathfrak{s}_E$ from its spinset $\Sigma_E$. 
If $E$ is a horizontal edge, so $\Sigma_E=\Sigma_{\h}$, we will say that the 
edge \textit{carries} the color $c$ if $\mathfrak{s}_E=c$. On
the other hand if $E$ is a vertical edge, we will say that $E$
\textit{carries} the color $c$ if 
$\mathfrak{s}_E=\gamma_1^{m_1}\cdots\gamma_r^{m_r}$ where $c=\gamma_i$ 
and $m_i>0$.

We next turn to the boundary conditions of the system.
By a \textit{flag} we mean a sequence $\mathbf{c}=(c_1,\cdots,c_r)$ of
colors, so that $c_i\in\{\gamma_1,\cdots,\gamma_r\}$. The colors may or
may not repeat. We will call $\mathbf{c}_0=(\gamma_1,\cdots,\gamma_r)$ the 
\textit{standard flag}. 

The (unfused) colored models, like the uncolored ones, are based on
an $r\times N$ grid as in Figure~\ref{fig:uncolored_model}, where 
$N$ is any sufficiently large integer. As in the uncolored
model, the columns are labeled in order increasing from
right to left. To describe the
boundary conditions, we need three pieces of data, a partition
$\lambda$ and two flags $\mathbf{c}$ and $\mathbf{d}$. It is
necessary that each color $\gamma_i$ occurs equal number of
times in $\mathbf{c}$ and $\mathbf{d}$, since otherwise there
will be no states with nonzero Boltzmann weight and the partition
function will vanish. We then use the following boundary
conditions. 

\begin{remark}
As in Section~\ref{sec:uncolored} we are taking $\lambda$ to be a
partition. However following Remark~\ref{rem:dominantpf} there is
no difficulty in extending to the case where $\lambda$ is a dominant
weight, which we will need in the last section, by adding columns
with negative labels. The renormalization by $(z_1\cdots z_r)^M$
explained in that remark works correctly and the
monostatic system $\mathfrak{S}_{\lambda, \mathbf{c},\mathbf{c}}(\mathbf{z};t)$ 
in Proposition~\ref{prop:monostatic} (Figure~\ref{fig:ground}) 
has weight $\mathbf{z}^\lambda t^{\ell(w)}$
even when $\lambda$ is allowed to have negative parts. In particular
Theorem~\ref{thm:demeval} below will remain valid when we apply it later.
\end{remark}

For the horizontal boundary spins on the left edge, we take
the boundary value $\plus$. For the horizontal spins on the right edge
in the $i$-th row, we take the boundary value $d_i$. The vertical
spins on the bottom row have boundary values the vacuum 
$\boson{0}:=\gamma_1^{0}\cdots\gamma_r^{0}$.
Now for the top boundary spins, in the $j$-th column we take
the boson $\gamma_1^{m_1}\cdots\gamma_r^{m_r}$, where $m_k$
is the number of pairs $(\lambda_i,c_i)$ with $\lambda_i=j$,
and $c_i=\gamma_i$.

%%%

To specify the Boltzmann weights, we will make use of a
phenomenon previously noted in~\cite{BBBGIwahori},
namely the \textit{monochrome factorization}, which we
now explain. There will be two equivalent versions of
the model, the colored or \textit{fused} model, based
on a grid with $rN$ vertices, and another \textit{unfused}
model, in which every vertex of the fused model is
replaced by $r$ vertices, and every vertical edge is
replaced by $r$ vertical edges. Thus there are a total
of $r^2N$ vertices in the unfused model. 

Therefore we begin by describing the vertices in the
unfused model. Such a vertex is described by two pieces of data: 
a spectral parameter $z_i$ and a color $c$. The vertical edges 
attached to the vertex of type $z_i,c$ can only carry the color $c$. 
For this reason the edge is described as \textit{monochrome}
(of color $c$). Its spinset $\Sigma^{\mon}_{\v,c}$ therefore
resembles the uncolored spinset $\Sigma^{\unc}_{\v}$ and is
in bijection with $\mathbb{N}$. The Boltzmann weights of
the admissible monochrome vertices are described in Figure~\ref{fig:monochrome}. 

Now we may describe the Boltzmann weights of the general
colored vertices by fusion, following the scheme in Figure~\ref{fig:fusion}.
Every vertex in the colored (fused) model is described by a single
spectral parameter $z_i$. Expanded, it is the fusion of $r$
distinct monochrome vertices of type $z_i,\gamma_r,\cdots,z_i,\gamma_1$
arranged in order from left to right. (Our convention is that $\gamma_r$
the smallest color, so this does not differ from the procedure in
\cite{BBBGIwahori} -- the apparent difference is just notational.)

Now let us describe the \textit{fusion} procedure used to define the Boltzmann weights
for the colored $P$- and $R$-models. Let 
$(a,b,c,d)\in\Sigma_{\h}^{\col}\times\Sigma_{\v}^{\col}\times\Sigma_{\h}^{\col}\times\Sigma_{\v}^{\col}$. 
We write $b=\gamma_1^{b_1}\cdots\gamma_r^{b_r}$ and
$d=\gamma_1^{d_1}\cdots\gamma_r^{d_r}$. Now consider the
configuration on the right side of Figure~\ref{fig:fusion}.
If there is a way of assigning boundary spins to the
interior edges such that the spins of the $r$ monochrome
vertices are all nonzero, that way is unique, and if
so we define the Boltzmann weight at the fused vertex
to be the product of the Boltzmann weights at the
unfused vertices. If there is no such way of assigning
interior spins, the Boltzmann weight is zero.

\begin{proposition}
The Yang-Baxter equation (Figure~\ref{fig:ybe})
is satisfied by the colored weights. For this, use the
R-matrix in Figure~\ref{fig:rmatrix} and the colored
weights defined above. The same R-matrix works for
both the $P$- and the $R$-weights.
\end{proposition}

\begin{proof}
We may replace the vertices by the corresponding
unfused vertices, and now make use of the auxiliary
Yang-Baxter equation in Figure~\ref{fig:ybeaux}. This may
be checked by consideration of the possible cases. The
auxiliary R-matrix $z_i,z_j,\gamma_k$ is defined in
Figure~\ref{fig:genrmatrix}. Note that if $k=r$ this
reduces to the R-matrix in Figure~\ref{fig:rmatrix}.
The R-matrix changes after moving past the monochrome
vertex, but when it has moved past all $r$ vertices,
it reverts to its original form.
\end{proof}

\tikzstyle{invisiblepath}=[draw=none, line width=1.5pt, spins]
\tikzstyle{path}=[line width=1.5pt, spins]
\tikzstyle{spins}=[every node/.style={spin}]
\tikzstyle{spin}=[circle, draw, fill=white, minimum size=14pt, inner sep=0pt, text=black]
\tikzstyle{halo}=[circle, fill=white, inner sep=0pt]
\tikzstyle{rectanglehalo}=[rectangle, fill=white, inner sep=0pt]

\usetikzlibrary{decorations.pathreplacing}

\newcommand{\icegrid}[2]{
    % #1: N columns (+1)
    % #2: rows

    % Vertical lines
    \foreach \i in {0,...,#1}{
        \pgfmathtruncatemacro\col{int(#1-\i)}
        \draw (2*\i+1,0) -- (2*\i+1,2*#2) node[spin, label=above:$\col$] {$+$};
    }

    % Horizontal lines
    \foreach \i in {1,...,#2}{
        \pgfmathtruncatemacro\row{int(#2-\i+1)}
        \draw (0,2*\i-1) node[spin, label=left:$\row$] {$+$} -- (2*#1+2, 2*\i-1);
    }

    % Underlying grid of plus spins
    \foreach \i in {0,...,#1}{
        \foreach \j in {1,...,#2}{
            \node[spin] at (2*\i+1, 2*\j-2) {$+$};
            \node[spin] at (2*\i+2, 2*\j-1) {$+$};
        }
    }
}

\newcommand{\icegridnorowlabels}[2]{
    % #1: N columns (+1)
    % #2: rows

    % Vertical lines
    \foreach \i in {0,...,#1}{
        \pgfmathtruncatemacro\col{int(#1-\i)}
        \draw (2*\i+1,0) -- (2*\i+1,2*#2) node[spin, label=above:$\col$] {$+$};
    }

    % Horizontal lines
    \foreach \i in {1,...,#2}{
        \pgfmathtruncatemacro\row{int(#2-\i+1)}
        \draw (0,2*\i-1) node[spin] {$+$} -- (2*#1+2, 2*\i-1);
    }

    % Underlying grid of plus spins
    \foreach \i in {0,...,#1}{
        \foreach \j in {1,...,#2}{
            \node[spin] at (2*\i+1, 2*\j-2) {$+$};
            \node[spin] at (2*\i+2, 2*\j-1) {$+$};
        }
    }
}

\newcommand{\icegridnocolumnlabels}[2]{
    % #1: N columns (+1)
    % #2: rows

    % Vertical lines
    \foreach \i in {0,...,#1}{
        \draw (2*\i+1,0) -- (2*\i+1,2*#2) node[spin] {$+$};
    }

    % Horizontal lines
    \foreach \i in {1,...,#2}{
        \pgfmathtruncatemacro\row{int(#2-\i+1)}
        \draw (0,2*\i-1) node[spin, label=left:$\row$] {$+$} -- (2*#1+2, 2*\i-1);
    }

    % Underlying grid of plus spins
    \foreach \i in {0,...,#1}{
        \foreach \j in {1,...,#2}{
            \node[spin] at (2*\i+1, 2*\j-2) {$+$};
            \node[spin] at (2*\i+2, 2*\j-1) {$+$};
        }
    }
}

\newcommand{\rowlabels}[2]{
   % #1: N columns (+1)
   % #2: rows
   \foreach \i in {0,...,#1}{
      \foreach \j in {1,...,#2}{
         \pgfmathtruncatemacro\row{int(#2-\j+1)}
         \node[halo] at (2*\i+1, 2*\j-1) {$\scriptstyle z_\row$};
      }
    }
}

%% Figure 10
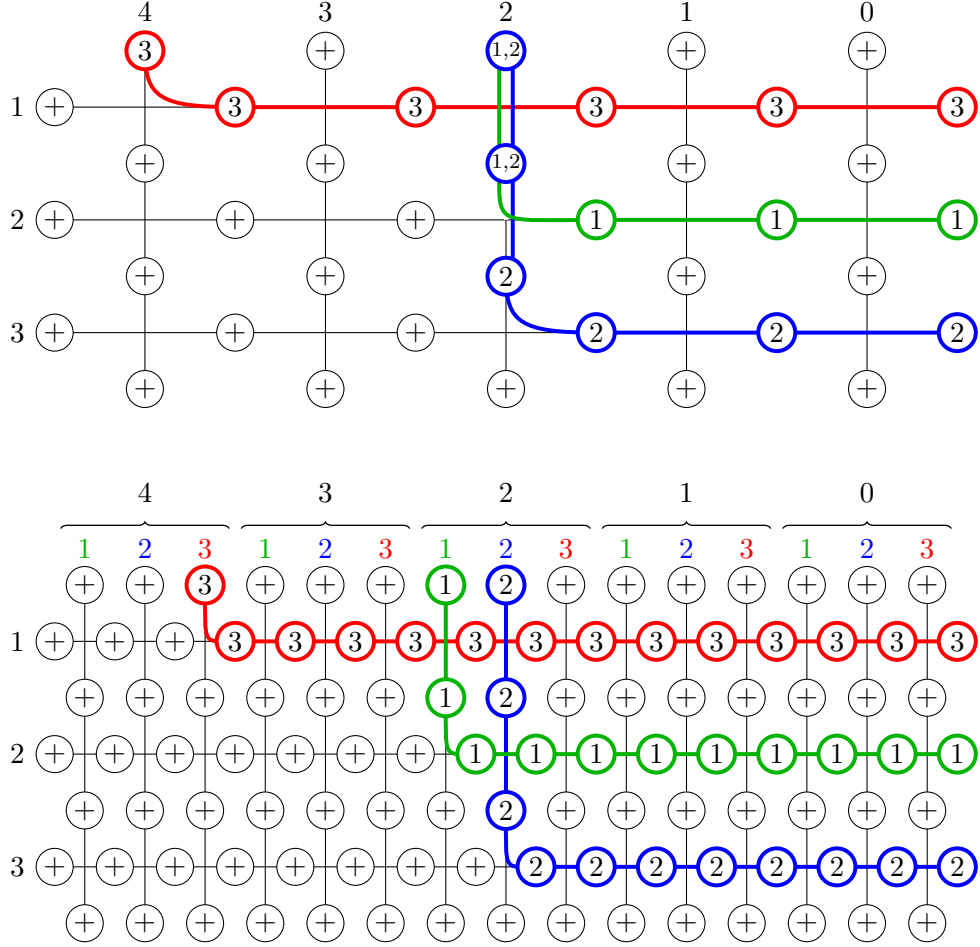
\begin{figure}[htb]
  \begin{centering}
  \begin{tikzpicture}[xscale=1.2, yscale=0.75, font=\small]
    \icegrid{4}{3}
    \draw[white,line width=0.5mm] (5,3) to (5,6);
    \draw [red, path] 
    (1,6) node {$3$} to [out=-90,in=180,looseness=1.5]
    (2,5) node {$3$} -- 
    (4,5) node {$3$} -- 
    (6,5) node {$3$} -- 
    (8,5) node {$3$} -- 
    (10,5) node {$3$};

    \draw [blue, path] (5+0.075,6) -- (5+0.075,2);
    \draw [darkgreen, path] (5-0.075,6) -- (5-0.075,4) to [out=-90,in=180,looseness=2.0] (6,3);
    
    \draw [blue, invisiblepath] 
    (5,6) node[draw=blue] {$\scriptstyle 1,2$} -- 
    (5,4) node[draw=blue] {$\scriptstyle 1,2$};

    \draw [blue, path]
    (5,2) node{$2$} to [out=-90,in=180,looseness=1.5] 
    (6,1) node{$2$} -- 
    (8,1) node{$2$} -- 
    (10,1) node{$2$};

    \draw[darkgreen, path] 
    (6,3) node {$1$} --
    (8,3) node {$1$} --
    (10,3) node {$1$};

  \end{tikzpicture} \\[2em]

\begin{tikzpicture}[xscale=0.4, yscale=0.75, font=\small]

      \icegridnocolumnlabels{14}{3}

      \draw[red, path]
      (5,6) node {$3$} to [out=-90,in=180,looseness=2]
      (6,5) node {$3$} --
      (8,5) node {$3$} --
      (10,5) node {$3$} -- 
      (12,5) node {$3$} -- 
      (14,5) node {$3$} -- 
      (16,5) node {$3$} -- 
      (18,5) node {$3$} -- 
      (20,5) node {$3$} -- 
      (22,5) node {$3$} -- 
      (24,5) node {$3$} -- 
      (26,5) node {$3$} -- 
      (28,5) node {$3$} -- 
      (30,5) node {$3$};

      \draw[blue, path]
      (15,6) node {$2$} --
      (15,4) node {$2$} --
      (15,2) node {$2$} to [out=-90,in=180,looseness=2]
      (16,1) node {$2$} --
      (18,1) node {$2$} --
      (20,1) node {$2$} --
      (22,1) node {$2$} --
      (24,1) node {$2$} --
      (26,1) node {$2$} --
      (28,1) node {$2$} --
      (30,1) node {$2$};

      \draw[darkgreen, path]
      (13,6) node {$1$} --
      (13,4) node {$1$} to [out=-90,in=180,looseness=2]
      (14,3) node {$1$} --
      (16,3) node {$1$} --
      (18,3) node {$1$} --
      (20,3) node {$1$} --
      (22,3) node {$1$} --
      (24,3) node {$1$} --
      (26,3) node {$1$} --
      (28,3) node {$1$} --
      (30,3) node {$1$};

      \foreach \i in {0,...,4}{
        \pgfmathtruncatemacro\col{int(4-\i)}
        \draw[line width=0.5pt, decoration={brace},decorate] (6*\i+0.2,7) -- node[above=6pt] {$\col$} (6*\i+6-0.2,7);
        \node[darkgreen] at (6*\i+1,6.65) {$1$};
        \node[blue] at (6*\i+3,6.65) {$2$};
        \node[red] at (6*\i+5,6.65) {$3$};

      }
  \end{tikzpicture}
  
  \end{centering}
  
\caption{The unique state of $\mathfrak{S}_{\lambda, \mathbf{c},\mathbf{c}}(\mathbf{z};t)$ 
for $G=\GL_3$ where $\lambda=(4,2,2)$, $\mathbf{c}=s_2\mathbf{c_0} = (R,G,B)$ and 
$\mathbf{c}_0=(\gamma_1,\gamma_2,\gamma_3)=(R,B,G)$. 
Top: the fused model. Bottom: the corresponding state in the unfused model.}
  \label{fig:ground}
  \end{figure}
% \input{figures/train}%

%% Figure 11
\begin{figure}[htb]
\begin{tikzpicture}[scale=0.7]
\begin{scope}[shift={(-1,0)}]
  \draw (0,1) to [out = 0, in = 180] (2,3) to (4,3);
  \draw (0,3) to [out = 0, in = 180] (2,1) to (4,1);
  \draw (3,0.25) to (3,3.75);
  \draw (7,0.25) to (7,3.75);
  \draw (6,1) to (8,1);
  \draw (6,3) to (8,3);
  \draw[fill=white] (-.2,1) circle (.5);
  \draw[fill=white] (-.2,3) circle (.5);
  \draw[line width=0.5mm,red,fill=white] (8.2,3) circle (.5);
  \draw[line width=0.5mm,blue,fill=white] (8.2,1) circle (.5);
  \node at (-.2,1) {$+$};
  \node at (-.2,3) {$+$};
  \node at (5,3) {$\cdots$};
  \node at (5,1) {$\cdots$};
  \draw[densely dashed] (3,3.75) to (3,4.25);
  \draw[densely dashed] (3,0.25) to (3,-0.25);
  \draw[densely dashed] (7,3.75) to (7,4.25);
  \draw[densely dashed] (7,0.25) to (7,-0.25);
  \node at (8.2,3) {$c_i$};
  \node at (8.2,1) {\scriptsize$c_{i+1}$};
\path[fill=white] (3,3) circle (.5);
\node at (3,3) {\scriptsize$z_{i+1}$};
\path[fill=white] (3,1) circle (.4);
\node at (3,1) {\scriptsize$z_i$};
\path[fill=white] (7,3) circle (.5);
\node at (7,3) {\scriptsize$z_{i+1}$};
\path[fill=white] (7,1) circle (.4);
\node at (7,1) {\scriptsize$z_i$};
\path[fill=white] (1,2) circle (.3);
\node at (1,2) {\scriptsize$z_{i+1},z_i$};
\end{scope}
\begin{scope}[shift={(1,-5.5)}]
  \draw (4,1) to (6,1) to [out = 0, in = 180] (8,3);
  \draw (4,3) to (6,3) to [out = 0, in = 180] (8,1);
  \draw[line width=0.5mm,red,fill=white] (8.2,3) circle (.5);
  \draw[line width=0.5mm,blue,fill=white] (8.2,1) circle (.5);
  \draw (0,1) to (2,1);
  \draw (0,3) to (2,3);
  \draw (5,0.25) to (5,3.75);
  \draw (1,0.25) to (1,3.75);
  \draw[fill=white] (-.2,1) circle (.5);
  \draw[fill=white] (-.2,3) circle (.5);

  \node at (3,1) {$\cdots$};
  \node at (3,3) {$\cdots$};

  \draw[densely dashed] (1,3.75) to (1,4.25);
  \draw[densely dashed] (1,0.25) to (1,-0.25);
  \draw[densely dashed] (5,3.75) to (5,4.25);
  \draw[densely dashed] (5,0.25) to (5,-0.25);
  \path[fill=white] (1,3) circle (.4);
  \node at (1,3) {\scriptsize$z_{i}$};
  \path[fill=white] (1,1) circle (.5);
  \node at (1,1) {\scriptsize$z_{i+1}$};

  \path[fill=white] (5,3) circle (.4);
  \node (a) at (5,3) {\scriptsize$z_{i}$};
  \path[fill=white] (5,1) circle (.5);
  \node at (5,1) {\scriptsize$z_{i+1}$};

  \path[fill=white] (7,2) circle (.3);
  \node at (7,2) {\scriptsize$z_{i+1},z_i$};
  \node at (8.2,1) {\scriptsize$c_{i+1}$};
  \node at (8.2,3) {$c_i$};
  \node at (-.2,1) {$+$};
  \node at (-.2,3) {$+$};
\end{scope}
\end{tikzpicture}
\caption{Proof of Proposition~\ref{prop:demtrain}. Top: the system $\mathfrak{S}_{\lambda,\mathbf{c},\mathbf{d}}(s_i\mathbf{z};t)$
with the R-matrix attached. Bottom: after using the Yang-Baxter equation.}
\label{fig:train}
\end{figure}
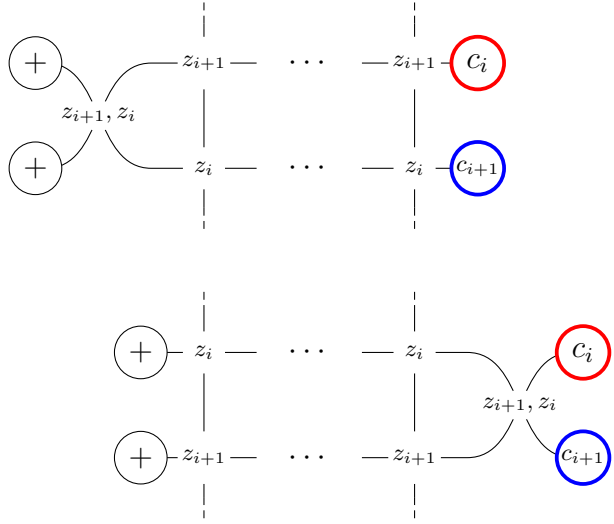

%% Figure 12
\begin{figure}[htb]
  \[\begin{array}{|c|c|c|}
  \hline
  \begin{tikzpicture}
    \draw[semithick] (-.75,-.75)--(.75,.75);
    \draw[semithick] (-.75,.75)--(.75,-.75);
    \draw[semithick,fill=white] (-.75,-.75) circle (.4);
    \draw[semithick,fill=white] (-.75,.75) circle (.4);
    \draw[semithick,fill=white] (.75,.75) circle (.4);
    \draw[fill=white] (.75,.-.75) circle (.4);
    \node at (-.75,-.75) {$+$};
    \node at (-.75,.75) {$+$};
    \node at (.75,.75) {$+$};
    \node at (.75,.-.75) {$+$};
    \path[semithick,fill=white] (0,0) circle (.4);
    \node at (0,0) {$z_{i+1},z_i$};
  \end{tikzpicture} &
  \begin{tikzpicture}
    \draw[line width=0.4mm,blue] (-.75,-.75)--(0,0);
    \draw[line width=0.4mm,red] (-.75,.75)--(0,0);
    \draw[line width=0.4mm,red] (.75,.75)--(0,0);
    \draw[line width=0.4mm,blue] (.75,-.75)--(0,0);
    \draw[line width=0.4mm,blue,fill=white] (-.75,-.75) circle (.4);
    \draw[line width=0.4mm,red,fill=white] (-.75,.75) circle (.4);
    \draw[line width=0.4mm,red,fill=white] (.75,.75) circle (.4);
    \draw[line width=0.4mm,blue,fill=white] (.75,.-.75) circle (.4);
    \node at (-.75,-.75) {$c_{i+1}$};
    \node at (-.75,.75) {$c_i$};
    \node at (.75,.75) {$c_i$};
    \node at (.75,.-.75) {$c_{i+1}$};
    \path[fill=white] (0,0) circle (.4);
    \node at (0,0) {$z_{i+1},z_i$};
  \end{tikzpicture} &
  \begin{tikzpicture}
    \draw[line width=0.4mm,blue] (-.75,-.75)--(0,0);
    \draw[line width=0.4mm,red] (-.75,.75)--(0,0);
    \draw[line width=0.4mm,blue] (.75,.75)--(0,0);
    \draw[line width=0.4mm,red] (.75,-.75)--(0,0);
    \draw[line width=0.4mm,blue,fill=white] (-.75,-.75) circle (.4);
    \draw[line width=0.4mm,red,fill=white] (-.75,.75) circle (.4);
    \draw[line width=0.4mm,blue,fill=white] (.75,.75) circle (.4);
    \draw[line width=0.4mm,red,fill=white] (.75,.-.75) circle (.4);
    \node at (-.75,-.75) {$c_{i+1}$};
    \node at (-.75,.75) {$c_i$};
    \node at (.75,.75) {$c_{i+1}$};
    \node at (.75,.-.75) {$c_i$};
    \path[fill=white] (0,0) circle (.4);
    \node at (0,0) {$z_{i+1},z_i$};
  \end{tikzpicture} \\
  \hline
  z_{i+1}-tz_i & (1-t)z_{i+1} & z_{i+1}-z_i\\
  \hline
  \end{array}\]
  \caption{R-matrix values needed in the proof of Proposition~\ref{prop:demtrain}. These may be read off from
  Figure~\ref{fig:rmatrix} bearing in mind that $c_i>c_{i+1}$.}
  \label{fig:rvalues}
\end{figure}

Let us say that a flag $\mathbf{c}=(c_1,\cdots,c_r)$ 
is \textit{proper} if its colors are all distinct.
If this is so, since we are working with a palette of
only $r$ colors, we must have $\mathbf{c}=w\mathbf{c}_0$
for some $w$ in the symmetric group $W=S_r$, and
similarly for $\mathbf{d}$.
For the remainder of the section, we consider models
$\mathfrak{S}_{\lambda,\mathbf{c},\mathbf{d}}(\mathbf{z};t)$
in which the flags $\mathbf{c}$ and $\mathbf{d}$ are
proper.

\begin{remark}
If $\mathbf{c}$ and $\mathbf{d}$ are proper, we will use the notation 
$\mathfrak{S}_{\lambda,\mathbf{c},\mathbf{d}}(\mathbf{z};t)$
to denote either model with $P$- or $R$- weights. This is justified 
because
\[Z_P(\mathfrak{S}_{\lambda,\mathbf{c},\mathbf{d}}(\mathbf{z};t))=
Z_R(\mathfrak{S}_{\lambda,\mathbf{c},\mathbf{d}}(\mathbf{z};t)).\]
\end{remark}

\begin{proof}
Indeed, with this assumption the Boltzmann weights of types
$C(n)$ and $D(n)$ in Figure~\ref{fig:monochrome} can only
appear with $n=0$, where they are the same in both models. So the Boltzmann
weights are the same at every vertex for the $P$- and $R$-models.
\end{proof}

If $\lambda=(\lambda_1,\cdots,\lambda_r)$ is a partition, we will denote
by $W_\lambda$ the stabilizer of $\lambda$ in~$W$. A system is
\textit{monostatic} if it has a single state, as in the next Proposition.

\begin{proposition}
\label{prop:monostatic}
Let $\mathbf{c}$ be a proper flag. Write $\mathbf{c}=w\mathbf{c}_0$.
Then
\[Z(\mathfrak{S}_{\lambda,\mathbf{c},\mathbf{c}}(\mathbf{z};t))=
t^{\ell(w)}\mathbf{z}^\lambda.\]
\end{proposition}

\begin{proof}
The assumption that $\mathbf{d}=\mathbf{c}$ guarantees that the system
$\mathfrak{S}_{\lambda,\mathbf{c},\mathbf{d}}(\mathbf{z};t)$ has a unique
state whose Boltzmann weight may be computed explicitly. See Figure~\ref{fig:ground}
for an example.
\end{proof}

In what follows, we use operators $\mathcal{L}_{i,t}$ from Section~\ref{sec:demazure}. Note that we use $t$ instead of $q$ to match the convention for the Hall-Littlewood polynomials. 

\begin{proposition}
\label{prop:demtrain}
Let $\mathbf{c},\mathbf{d}$ be a proper flags, and write $\mathbf{d}=w\mathbf{c}_0$
with $w\in W$.  Let $s=s_i$ be the simple reflection $(i,i+1)\in W$.
Then
\[Z(\mathfrak{S}_{\lambda,\mathbf{c},s_i\mathbf{d}}(\mathbf{z};t))=
\left\{
\begin{array}{cl}
\mathcal{L}_{i,t}Z(\mathfrak{S}_{\lambda,\mathbf{c},\mathbf{d}}(\mathbf{z};t)) &\text{if $s_iw>w$,}\\
\mathcal{L}_{i,t}^{-1}Z(\mathfrak{S}_{\lambda,\mathbf{c},\mathbf{d}}(\mathbf{z};t)) &\text{if $s_iw<w$.}
\end{array}\right.\]
\end{proposition}

\begin{proof}
To simplify the notation
we will write $\phi_w(\mathbf{z})=Z(\mathfrak{S}_{\lambda,\mathbf{c},w\mathbf{c}_0}(\mathbf{z};t)$,
so what we are trying to prove is that if $s_iw>w$ then $\phi_{s_iw}=\mathcal{L}_{i,t}\phi_w$.
We consider the system $\mathfrak{S}_{\lambda,\mathbf{c},\mathbf{d}}(s_i\mathbf{z};t)$
where the effect of the $s_i$ is to switch $z_i$ and $z_{i+1}$. We attach
the R-matrix to the left as in Figure~\ref{fig:train}. The meaning of the
assumption $s_iw>w$ is that the color $c_i=\gamma_{w^{-1}(i)}$ is
$>c_{i+1}$. These colors are red and blue, respectively in
Figure~\ref{fig:train}. Referring to the figure the R-matrix on
the top has only one possible state, while the R-matrix
at the bottom has two. The values
we need are in Figure~\ref{fig:rvalues}. We obtain
\[(z_{i+1}-tz_i)\phi_w(s_i\mathbf{z})=
(1-t)z_{i+1}\phi_w(\mathbf{z}) + (z_{i+1}-z_i)\phi_{w}(s_i\mathbf{z}).\]
Rearranging,
\[\phi_{s_iw}(\mathbf{z})=\frac{\phi_w(\mathbf{z})-\phi_{w}(s_i\mathbf{z})
-t(\phi_w(\mathbf{z})-\mathbf{z}^{\alpha_i}\phi_w(s_i\mathbf{z}))}
{\mathbf{z}^{\alpha_i}-1},\qquad \mathbf{z}^{\alpha_i}=z_i/z_{i+1},\]
which is $\mathcal{L}_{i,t}\phi_w(\mathbf{z})$.
\end{proof}

\begin{theorem}
  \label{thm:demeval}Let $w \in W$. Then
  \[ Z (\mathfrak{S}_{\lambda, \mathbf{c}_0, w\mathbf{c}_0}
     (\mathbf{z};t)) =\mathcal{L}_{w, t} \mathbf{z}^{\lambda} . \]
  More generally, if $w,y\in W$, then
  \[Z(\mathfrak{S}_{\lambda,y\mathbf{c}_0,w\mathbf{c}_0}(\mathbf{z};t))=
  \tau_{w,y}^\lambda(\mathbf{z};t)\]
  where $\tau_{w,y}^\lambda$ is as in Theorem~\ref{thm:taurecurse}.
\end{theorem}

\begin{proof}
  This follows immediately from the last two propositions and Theorem~\ref{thm:taurecurse}.
\end{proof}

We note that
$\mathcal{L}_{w,t}\mathbf{z}^\lambda$, up to a scalar multiple, matches the 
non-symmetric Hall-Littlewood polynomial $E_{w(\lambda)}(\infty, t)$
in the notation of \cite{IonStandard}, Corollary~3.8.

\section{The local lifting property\label{sec:local_lifting}}

% \input {figures/local-lifting}

%% Figure 13
\begin{figure}[htb]
\[\begin{array}{|c|c||c|}
\hline
\multicolumn{2}{|c||}
{\text{before applying $p''$}} & \text{after $p''$}\\
\hline
% 1
\begin{tikzpicture}[scale=.6,every node/.style={scale=.7}]
\draw[blue, line width=.5mm] (0,1)--(4,1);
\draw[blue, line width=.5mm] (1,0)--(1,2);
\draw[red, line width=.5mm] (3,0)--(3,2);
\node at (.7,0) {$n$};
\node at (.7,1.8) {$n$};
\node at (2.6,0) {$m$};
\node at (2.6,1.8) {$m$};
\node at (1,2.5) {$c_i$};
\node at (3,2.5) {$c_{i+1}$};
\end{tikzpicture} 
&
\begin{tikzpicture}[scale=.6,every node/.style={scale=.6}]
\draw[blue, line width=.5mm] (0,1)--(1,1);
\draw[semithick] (1,1)--(3,1);
\draw[red, line width=.5mm] (3,1)--(4,1);
\draw[blue, line width=.5mm] (1,0)--(1,2);
\draw[red, line width=.5mm] (3,0)--(3,2);
\node at (.6,1.8) {$n$};
\node at (.4,0) {$n\!+\!1$};
\node at (2.6,1.8) {$m$};
\node at (2.4,0) {$m\!-\!1$};
\node at (1,2.5) {$c_i$};
\node at (3,2.5) {$c_{i+1}$};
\end{tikzpicture}
&
\begin{tikzpicture}[scale=.6,every node/.style={scale=.6}]
\draw (1,0)--(1,2);
\draw[red, line width=.5mm] (3,0)--(3,2);
\draw[red, line width=.5mm] (0,1)--(4,1);
% \draw (3,0)--(3,2);
\node at (2.4,0) {$n\!+\!m$};
\node at (2.4,1.8) {$n\!+\!m$};
\node at (.6,1.8) {$0$};
\node at (1,2.5) {$c_i$};
\node at (3,2.5) {$c_{i+1}$};
\end{tikzpicture}\\
\hline
zt^m & z(1-t^m) & z\\
\hline
% 2
\multicolumn{2}{|c||}
{\begin{tikzpicture}[scale=.6,every node/.style={scale=.7}]
\draw[red, line width=.5mm] (0,1)--(4,1);
\draw[blue, line width=.5mm] (1,0)--(1,2);
\draw[red, line width=.5mm] (3,0)--(3,2);
\node at (.7,0) {$n$};
\node at (.7,2) {$n$};
\node at (2.6,0) {$m$};
\node at (2.6,2) {$m$};
\end{tikzpicture}} 
&
\begin{tikzpicture}[scale=.6,every node/.style={scale=.6}]
\draw (1,0)--(1,2);
\draw[red, line width=.5mm] (3,0)--(3,2);
\draw[red, line width=.5mm] (0,1)--(4,1);
% \draw (3,0)--(3,2);
\node at (.6,1.8) {$0$};
\node at (2.4,0) {$n\!+\!m$};
\node at (2.4,2) {$n\!+\!m$};
\end{tikzpicture}
\\
\hline
% 3
\multicolumn{2}{|c||}{z} & z\\
\hline
\multicolumn{2}{|c||}
{\begin{tikzpicture}[scale=.6,every node/.style={scale=.7}]
\draw[darkgreen, line width=.5mm] (0,1)--(4,1);
\draw[blue, line width=.5mm] (1,0)--(1,2);
\draw[red, line width=.5mm] (3,0)--(3,2);
\node at (.7,0) {$n$};
\node at (.7,2) {$n$};
\node at (2.6,0) {$m$};
\node at (2.6,2) {$m$};
\end{tikzpicture}} 
&
\begin{tikzpicture}[scale=.6,every node/.style={scale=.6}]
\draw (1,0)--(1,2);
\draw[red, line width=.5mm] (3,0)--(3,2);
\draw[darkgreen, line width=.5mm] (0,1)--(4,1);
% \draw (3,0)--(3,2);
\node at (0.4,2) {$0$};
\node at (2.4,0) {$n\!+\!m$};
\node at (2.4,2) {$n\!+\!m$};
\end{tikzpicture}
\\
\hline
\multicolumn{2}{|c||}
{\begin{array}{cl}
t^{n+m} & \text{if ${\green G} < \red{R},\blue{B}$,}\\
1 & \text{if ${\green G} > \red{R},\blue{B}$}\end{array}} &
{\begin{array}{cl}
t^{n+m} & \text{if ${\green G} < \red{R},\blue{B}$,}\\
1 & \text{if ${\green G} > \red{R},\blue{B}$}\end{array}}
\\
\hline
% 4
\multicolumn{2}{|c||}
{\begin{tikzpicture}[scale=.6,every node/.style={scale=.7}]
\draw[blue, line width=.5mm] (0,1)--(1,1);
\draw[semithick] (1,1)--(4,1);
\draw[blue, line width=.5mm] (1,0)--(1,2);
\draw[red, line width=.5mm] (3,0)--(3,2);
\node at (.4,0) {$n\!+\!1$};
\node at (.7,2) {$n$};
\node at (2.6,0) {$m$};
\node at (2.6,2) {$m$};
\end{tikzpicture}} 
&
\begin{tikzpicture}[scale=.6,every node/.style={scale=.6}]
\draw (1,0)--(1,2);
\draw[red, line width=.5mm] (3,0)--(3,2);
\draw[semithick] (1,1)--(4,1);
\draw[red, line width=.5mm] (0,1)--(3,1);
% \draw (3,0)--(3,2);
\node at (0.6,0) {$0$};
\node at (2.2,0) {$n\!+\!m\!+\!1$};
\node at (2.4,2) {$n\!+\!m$};
\end{tikzpicture}
\\\hline
\multicolumn{2}{|c||}
{z(1-t)} & {z(1-t)}\\
\hline
% 5
\multicolumn{2}{|c||}
{\begin{tikzpicture}[scale=.6,every node/.style={scale=.7}]
\draw[red, line width=.5mm] (0,1)--(3,1);
\draw[semithick] (3,1)--(4,1);
\draw[blue, line width=.5mm] (1,0)--(1,2);
\draw[red, line width=.5mm] (3,0)--(3,2);
\node at (.7,0) {$n$};
\node at (.7,2) {$n$};
\node at (2.4,0) {$m\!+\!1$};
\node at (2.6,2) {$m$};
\end{tikzpicture}} 
&
\begin{tikzpicture}[scale=.6,every node/.style={scale=.6}]
\draw (1,0)--(1,2);
\draw[red, line width=.5mm] (3,0)--(3,2);
\draw[semithick] (1,1)--(4,1);
\draw[red, line width=.5mm] (0,1)--(3,1);
% \draw (3,0)--(3,2);
\node at (0.6,0) {$0$};
\node at (2.4,2) {$n\!+\!m$};
\node at (2.2,0) {$n\!+\!m\!+1$};
\end{tikzpicture}
\\\hline
\multicolumn{2}{|c||}
{z(1-t)} & {z(1-t)}\\
\hline
% 6
\begin{tikzpicture}[scale=.6,every node/.style={scale=.7}]
\draw[semithick] (0,1)--(1,1);
\draw[blue, line width=.5mm] (1,1)--(4,1);
\draw[blue, line width=.5mm] (1,0)--(1,2);
\draw[red, line width=.5mm] (3,0)--(3,2);
\node at (.4,0) {$n\!-\!1$};
\node at (.6,2) {$n$};
\node at (2.6,0) {$m$};
\node at (2.6,2) {$m$};
\end{tikzpicture} 
&
\begin{tikzpicture}[scale=.6,every node/.style={scale=.7}]
\draw[semithick] (0,1)--(3,1);
\draw[red, line width=.5mm] (3,1)--(4,1);
\draw[blue, line width=.5mm] (1,0)--(1,2);
\draw[red, line width=.5mm] (3,0)--(3,2);
\node at (.7,0) {$n$};
\node at (.7,2) {$n$};
\node at (2.4,0) {$m\!-\!1$};
\node at (2.6,2) {$m$};
\end{tikzpicture} 
&
\begin{tikzpicture}[scale=.6,every node/.style={scale=.6}]
\draw (1,0)--(1,2);
\draw[red, line width=.5mm] (3,0)--(3,2);
\draw[red, line width=.5mm] (3,1)--(4,1);
\draw[semithick] (0,1)--(3,1);
% \draw (3,0)--(3,2);
\node at (0.6,0) {$0$};
\node at (2.2,0) {$n\!+\!m$};
\node at (2.4,2) {$n\!+\!m$};
\end{tikzpicture}\\
\hline
\frac{1-t^n}{1-t}t^m & \frac{1-t^m}{1-t} & \frac{1-t^{n+m}}{1-t}\\
\hline
% 7
\multicolumn{2}{|c||}
{\begin{tikzpicture}[scale=.6,every node/.style={scale=.7}]
\draw[semithick] (0,1)--(4,1);
\draw[red, line width=.5mm] (3,0)--(3,2);
\draw[blue, line width=.5mm] (1,0)--(1,2);
\node at (.6,0) {$n$};
\node at (.6,2) {$n$};
\node at (2.6,0) {$m$};
\node at (2.6,2) {$m$};
\end{tikzpicture}} 
&
\begin{tikzpicture}[scale=.6,every node/.style={scale=.6}]
\draw (1,0)--(1,2);
\draw[red, line width=.5mm] (3,0)--(3,2);
% \draw (3,0)--(3,2);
\draw[semithick] (0,1)--(4,1);
\node at (0.6,0) {$0$};
\node at (2.4,0) {$n\!+\!m$};
\node at (2.4,2) {$n\!+\!m$};
\end{tikzpicture}\\
\hline
\multicolumn{2}{|c||} 1 & 1\\
\hline
\end{array}\]
\caption{Verifying the local lifting property for the maps $p''_{i,\h}$ and $p''_{i,\v}$. The figure shows the \textit{unfused} vertices in the columns of the monochrome edges for $\gamma_i$, $\gamma_{i+1}$. The monochrome vertex for the color $\gamma_{i+1}$ (blue) is on the left, and the monochrome vertex for the color $\gamma_i$ (red) is on the right. The maps $p_{i,\h}''$ and $p_{i+i,\v}''$ replace the color $\gamma_{i+1}$ with $\gamma_i$, and the case-by-case analysis confirms that this procedure does not affect the Boltzmann weights. In the third row, green is any color that is not $\gamma_i$ or $\gamma_{i+1}$.} 
\label{fig:local_lifting}
\end{figure}

In this section we will prove the local lifting property for the $R$-weights,
which is a relationship between the colored and uncolored models.
In~\cite{BorodinWheelerColored}, Proposition~2.4.2, this fact is
called \textit{color-blindness}.
The local lifting property is the analog of Properties~A and~B in
{\cite{BBBGIwahori}}, Section~8. As in that paper, the local lifting property
has implications for the (global) partition functions. In particular, we will
here apply it to reprove Korff's evaluation of the uncolored lattice models.

The $P$-weights also have a lifting property, but it is less strictly local,
since it concerns column transfer matrices instead of individual vertices. The
local lifting property as formulated here is strictly for the $R$- weights.

To formulate the local lifting property for the $R$-weights, we make use of
both the (fused) colored and the uncolored spinsets for horizontal and
vertical edges. We define maps $p_{\h} : \Sigma^{\col}_{\h}
\longrightarrow \Sigma^{\unc}_{\h}$ and $p_{\v} :
\Sigma^{\col}_{\v} \longrightarrow
\Sigma^{\unc}_{\v}$. Specifically, $p_{\h} :
\Sigma^{\col}_{\h} \longrightarrow \Sigma^{\unc}_{\h}$
is the map
\[ p_{\h} (+) = +, \qquad p_{\h} (c_i) = -, \]
and $p_{\v} : \Sigma^{\col}_{\v} \longrightarrow
\Sigma^{\unc}_{\v}$ is the map that sends
\[ \text{$\gamma_1^{m_1} \cdots \gamma_r^{m_r} \in \Sigma^{\col}_{\v}
   \quad$to} \quad \sum_{i = 1}^r m_i \in \Sigma_{\h}^{\unc} . \]
Intuitively, these maps just change every colored boson to the uncolored boson
represented by $-$ in the uncolored model.

\begin{proposition}
  \label{prop:loclif}
  Let $(A, B, C, D) \in \Sigma_{\h}^{\unc} \times
  \Sigma_{\v}^{\unc} \times \Sigma_{\h}^{\unc} \times
  \Sigma_v^{\unc}$, and let $(a, b) \in \Sigma_{\h}^{\col}
  \times \Sigma_{\v}^{\col}$ such that $p_{\h} (a) = A$ and
  $p_{\v} (b) = B$. Then

\begin{equation}
\label{eq:loclif}
\beta_R^{\unc}\left(\vcenter{\hbox to 80pt{
\begin{tikzpicture}
\draw (1,0)--(1,2);
\draw (0,1)--(2,1);
\draw[fill=white] (1,0) circle (.3);
\draw[fill=white] (1,2) circle (.3);
\draw[fill=white] (0,1) circle (.3);
\draw[fill=white] (2,1) circle (.3);
\node at (0,1) {$A$};
\node at (1,2) {$B$};
\node at (2,1) {$C$}; 
\node at (1,0) {$D$}; 
\end{tikzpicture}}}\right) =
\sum_{\substack{(c,d)\in\Sigma_{\h}^{\col}\times\Sigma_{\v}^{\col}\\
p_{\h}(c)=C\\p_{\v}(d)=D}}
\beta_R^{\col}\left(\vcenter{\hbox to 80pt{
\begin{tikzpicture}
\draw (1,0)--(1,2);
\draw (0,1)--(2,1);
\draw[fill=white] (1,0) circle (.3);
\draw[fill=white] (1,2) circle (.3);
\draw[fill=white] (0,1) circle (.3);
\draw[fill=white] (2,1) circle (.3);
\node at (0,1) {$a$};
\node at (1,2) {$b$};
\node at (2,1) {$c$}; 
\node at (1,0) {$d$}; 
\end{tikzpicture}}}\right) .
\end{equation}
\end{proposition}

\begin{proof}
  First, we note that the uncolored $R$-weights actually coincide with the
  colored $R$-weights using only one color $c$. So we may work instead with
  maps $p_{\h}' : \Sigma^{\col}_{\h} \longrightarrow
  \Sigma^{\col}_{\h}$ and $p_{\v}' :
  \Sigma_{\v}^{\col} \longrightarrow
  \Sigma_{\v}^{\col}$ that replace every color by $\gamma_1$. Thus
  \[ p_{\h}' (+) = +, \qquad p_{\h}' (\gamma_i) = \gamma_1, \]
  and $p'_{\v}$
  \[ \text{$\gamma_1^{m_1} \cdots \gamma_r^{m_r} \in \Sigma^{\col}_{\v}
     \quad$to} \quad \gamma_1^m \in \Sigma_{\h}^{\unc}, \qquad m =
     \sum_{i = 1}^r c. \]
  We may verify by comparing the weights in Figure~\ref{fig:uncolored_weights}
  with the colored weights using only one color that if $(a, b, c, d) \in
  \Sigma_{\h}^{\col} \times \Sigma_{\v}^{\col} \times
  \Sigma_{\h}^{\col} \times \Sigma_v^{\col}$ then the
  Boltzmann weights of $\left( p_{\h} (a), p_{\v} (b),
  p_{\h} (c), p_{\text{} v} (d) \right) \in
  \Sigma_{\h}^{\unc} \times \Sigma_{\v}^{\unc} \times
  \Sigma_{\h}^{\unc} \times \Sigma_v^{\unc}$ and $\left(
  p'_{\h} (a), p'_{\v} (b), p'_{\h} (c), p'_{\text{} v} (d)
  \right) \in \Sigma_{\h}^{\col} \times
  \Sigma_{\v}^{\col} \times \Sigma_{\h}^{\col} \times
  \Sigma_v^{\col}$ are the same. (To confirm this, observe that in the
  unfused weights corresponding to the fused weights $p'_{\v} (b)$ and
  $p'_{\v} (d)$ we may optionally discard the monochrome vertices for
  colors $\gamma_i$ with $i \neq 1$, since the Boltzmann weights at these vertices
  all equal~1.)
  
  Now $p'$ maps may themselves be factored into maps that change only one
  color at a time. Thus we consider the maps 
  $p_{i, \h}' : \Sigma^{\col}_{\h} \longrightarrow
  \Sigma^{\col}_{\h}$ and 
  $p_{i, \v}'' : \Sigma_{\v}^{\col} \longrightarrow \Sigma_{\v}^{\col}$ that 
  replace the color $c_{i + 1}$ by $\gamma_i$,
  and leave the other colors $c_j$ unaltered.
  
% \clearpage % to prevent Figure 10 from floating to the end of the paper

  The maps $p_{i, \h}''$ and $p''_{i, \v}$ have the following local
  lifting property (see Figure~\ref{fig:local_lifting} for illustration): 
  if $(A, B, C, D) \in \Sigma_{\h}^{\col} \times
  \Sigma_{\v}^{\col} \times \Sigma_{\h}^{\col} \times
  \Sigma_{\v}^{\col}$ and $(a, b) \in
  \Sigma_{\h}^{\col} \times \Sigma_{\v}^{\col}$ such
  that $p''_{i, \h} (a) = A$ and $p''_{i, \v} (b) = B$, then
\[\beta_R^{\col}\left(\vcenter{\hbox to 80pt{
\begin{tikzpicture}
\draw (1,0)--(1,2);
\draw (0,1)--(2,1);
\draw[fill=white] (1,0) circle (.3);
\draw[fill=white] (1,2) circle (.3);
\draw[fill=white] (0,1) circle (.3);
\draw[fill=white] (2,1) circle (.3);
\node at (0,1) {$A$};
\node at (1,2) {$B$};
\node at (2,1) {$C$}; 
\node at (1,0) {$D$}; 
\end{tikzpicture}}}\right) =
\sum_{\substack{(c,d)\in\Sigma_{\h}^{\col}\times\Sigma_{\v}^{\col}\\
p''_{i,\h}(c)=C\\p''_{i,\v}(d)=D}}
\beta_R^{\col}\left(\vcenter{\hbox to 80pt{
\begin{tikzpicture}
\draw (1,0)--(1,2);
\draw (0,1)--(2,1);
\draw[fill=white] (1,0) circle (.3);
\draw[fill=white] (1,2) circle (.3);
\draw[fill=white] (0,1) circle (.3);
\draw[fill=white] (2,1) circle (.3);
\node at (0,1) {$a$};
\node at (1,2) {$b$};
\node at (2,1) {$c$}; 
\node at (1,0) {$d$}; 
\end{tikzpicture}}}\right) . \]
  
  Now $p'_{\h} = p''_{1, \h} \circ \ldots \circ p''_{r - 1,
  \h}$ and $p'_{\v} = p''_{1, \v} \circ \ldots \circ p''_{r
  - 1, \v}$ has a local lifting property. Replacing the color $c_1$ by
  $+$ and (in the monochrome model) deleting the vertices labeled $c_2,
  \cdots, c_r$ we obtain the local lifting property for $p_{\h}$ and
  $p_{\v}$.
\end{proof}

The local lifting result implies the following \textit{global lifting}
property for partition functions. Using the maps $p_{\h}$ and $p_{\v}$
we may construct a map $\mathfrak{s}\mapsto p(\mathfrak{s})$ from
states of the colored model $\mathfrak{S}_{\lambda,\mathbf{c},\mathbf{d}}(\mathbf{z};t)$,
to the uncolored model $\mathfrak{S}_\lambda(\mathbf{z};t)$.
We simply apply $p_{\h}$ and $p_{\v}$ to the spins of every edge in $\mathfrak{s}$,
obtaining a state of $\mathfrak{S}_\lambda(\mathbf{z};t)$.

If $\mathbf{c}$ is a flag let
\[\mathfrak{S}_{\lambda,\mathbf{c}}^{\col}(\mathbf{z};t)=\bigsqcup_{\mathbf{d}}
\mathfrak{S}_{\lambda,\mathbf{c},\mathbf{d}}^{\col}(\mathbf{z};t).\]
Thus the set $\mathfrak{S}_{\lambda,\mathbf{c}}^{\col}(\mathbf{z};t)$
is the disjoint union of all states with prescribed flag on the upper
boundary, but arbitrary flag on the right boundary.

\begin{proposition}
\label{prop:glolif}
Let $\mathfrak{s}_0$ be a state of $\mathfrak{S}_\lambda(\mathbf{z};t)$.
Let $\mathbf{c}$ be a flag.
\[\beta_R^{\unc}(\mathfrak{s}_0)=
\sum_{\substack{s\in\mathfrak{S}^{\col}_{\lambda,\mathbf{c}}(\mathbf{z};t)\\
p(\mathbf{s})=\mathfrak{s}_0}}\beta_R^{\col}(\mathbf{s}).\]
\end{proposition}

\begin{proof}
A formal proof follows the method of Lemma~8.5 of~\cite{BBBGIwahori}, and
we refer to that proof for details, giving here a brief intuitive explanation.
We may construct all the liftings $\mathfrak{s}$ of the state $\mathfrak{s}_0$
to $\mathfrak{S}_{\lambda,\mathbf{c}}^{\col}(\mathbf{z};t)$ algorithmically
by assigning colored spins to the edges right and below each vertex, visiting
the vertices in order, row by row, from left to right. In this procedure,
as each vertex is considered, the spins to the left and above the vertex are
known either because of the condition that 
$\mathfrak{s}\in\mathfrak{S}_{\lambda,\mathbf{c}}^{\col}(\mathbf{z};t)$,
which determines the left and top boundary spins, or because the spins
above and left of the vertex have already been assigned at an earlier
stage of the algorithm. Thus at the vertex under consideration if the
spins of the adjacent edges are labeled $a,b,c,d$ as in Proposition~\ref{prop:loclif},
the vertices $a$, $b$ are known, while $c$ and $d$ must satisfy
$p_{\h}(c)=C$, $p_{\v}(d)=D$, where $C$ and $D$ are the corresponding
spins of the state $\mathfrak{s}_0$, due to the requirement that
$p(\mathfrak{s})=\mathfrak{s}_0$. We see that the possible values for
the spins at the given vertex are exactly those on the right-hand side
of (\ref{eq:loclif}), and taking this into account, the sum of the
Boltzmann weights of the possible configurations at the vertex equals
the Boltzmann weight of the vertex in the corresponding uncolored
state $\mathfrak{s}_0$. The result follows from this observation.
See Section~8 of~\cite{BBBGIwahori} for a formal argument. 
\end{proof}

\begin{corollary}
\label{cor:global_lifting}
Let $\mathbf{c}$ be any flag. Then
\[Z_R(\mathfrak{S}^{\unc}_\lambda(\mathbf{z};t))=
\sum_{\mathbf{d}} Z_R(\mathfrak{S}^{\col}_{\lambda,\mathbf{c},\mathbf{d}}(\mathbf{z};t)).\]
\end{corollary}

\noindent
We emphasize that this statement is for the $R$-models only,
and fails for the $P$-models, since the local lifting
property is only valid for the $R$-models.

\begin{proof}
This follows from Proposition~\ref{prop:glolif} by summing over all
states $\mathfrak{s}_0$ of $\mathfrak{S}^{\unc}_\lambda(\mathbf{z};t)$.
\end{proof}

\begin{theorem}\label{thm:uncoloredpf}
  The partition functions of these systems equal the
  Hall-Littlewood $P$- and $R$-polynomials (\cite{MacdonaldBook}, Chapter~III):
  \begin{equation}
  \label{eq:rpeval}
  Z_P(\mathfrak{S}_{\lambda}^{\unc} (\mathbf{z};t)) = 
  P_{\lambda} (z_1, \cdots, z_r; t), \qquad 
  Z_R (\mathfrak{S}_{\lambda}^{\unc} (\mathbf{z}; t)) = 
  R_{\lambda} (z_1, \cdots, z_r; t) .
  \end{equation}
\end{theorem}
\begin{proof}
For the $P$-system, this is due to Korff~{\cite{KorffVerlinde}}, and
for the $R$-system, we will give a new proof. 
But note that by Proposition~\ref{prop:uncoloredpf}, and the fact that
in Macdonald's definition (\cite{MacdonaldBook}, Section~III.1)
\[R_\lambda(\mathbf{z};t)=v_\lambda(t)\,P_\lambda(\mathbf{z};t),\]
the two formulas are equivalent, and so we have new proofs of both evaluations.

Let us apply Corollary~5.3 in the special case where the flag $\mathbf{c}$
equals $\mathbf{c}_0 = (\gamma_1, \cdots, \gamma_r)$. Then using
Theorem~\ref{thm:demeval} and Proposition~\ref{prop:sphericaldem},
\[ Z_R (\mathfrak{S}^{\unc}_{\lambda}(\mathbf{z};t)) = \sum_{w \in W} Z_R
   (\mathfrak{S}^{\col}_{\lambda, \mathbf{c}_0, w\mathbf{c}_0}(\mathbf{z};t)) =
   \sum_{w \in W} \mathcal{L}_{w, t} \mathbf{z}^{\lambda} = R_{\lambda}
   (\mathbf{z}; t), \]
which is the second equation in (\ref{eq:rpeval}). We have already shown that
the two evaluations are equivalent.
\end{proof}

\section{Spherical models of $p$-adic representations}
Let $F$ be a nonarchimedean local field with ring $\mathfrak{o}$ of integers,
and prime ideal $\mathfrak{p}$ with a chosen generator $\varpi$. Let $q$
be the cardinality of the residue field $\mathfrak{o}/\mathfrak{p}$. Let $G (F) =
\GL_r (F)$, with $K = \GL_r (\mathfrak{o})$ a maximal compact
subgroup. We will denote by $J$ the Iwahori subgroup of $K$ consisting of
elements that are upper triangular modulo $\mathfrak{p}$. It will be
convenient to identify the Weyl group $W \cong S_r$ of $G$ with the subgroup
of $G$ consisting of permutation matrices. With this choice, every element of
$W$ is actually in~$K$.

Let $T$ be the {\textit{standard maximal torus}} of the affine algebraic group
$G = \GL_r$ so $T (F)$ consists of diagonal matrices in $G (F)$. Let $B$
be the {\textit{standard Borel subgroup}} of upper triangular matrices. Let $N$
be the unipotent radical of $B$ consisting of unipotent upper triangular
matrices, so $B = T N$.

The Langlands dual group $\hat{G}$ is just $\GL_r (\mathbb{C})$. We will
denote $\Lambda =\mathbb{Z}^r$; it may be thought of as the weight lattice of
$\hat{G}$. \ Let $\hat{T}$ be the diagonal subgroup of the Langlands dual
group $\GL_r (\mathbb{C})$. We will identify $\hat{T} =
(\mathbb{C}^{\times})^r$ in the obvious way, so an element $\mathbf{z} \in
\hat{T}$ can be identified with a tuple $\mathbf{z}= (z_1, \cdots, z_r)$
with $z_i \in \mathbb{C}^{\times}$. If $\mathbf{z} \in \hat{T}$ and $\lambda
\in \Lambda$ we will denote by $\mathbf{z}^{\lambda} = \prod
z_i^{\lambda_i}$. In particular if $\Phi$ is the root system of $\hat{G}$, and
if $\alpha_1, \cdots, \alpha_{r - 1}$ are the simple roots, then
$\mathbf{z}^{\alpha_i} = z_i / z_{i + 1}$.

The lattice $\Lambda$ may also be thought of as the coweight lattice of
$\GL_r$, which is isomorphic to $T(F)/T(\mathfrak{o})$. 
Thus if $\lambda \in \Lambda$ let $\varpi^{\lambda} =
\operatorname{diag} (\varpi^{\lambda_1}, \cdots, \varpi^{\lambda_r}) \in T (F)$. A
quasicharacter of $T (F)$ is {\textit{unramified}} if it is trivial on $T
(\mathfrak{o})$. Given $\mathbf{z} \in \hat{T}$ there is a unique unramified
quasicharacter $\chi_{\mathbf{z}} : T (F) \longrightarrow
\mathbb{C}^{\times}$ such that $\chi_{\mathbf{z}} (\varpi^{\lambda})
=\mathbf{z}^{\lambda}$. We may extend it to a homomorphism $B (F)
\longrightarrow \mathbb{C}^{\times}$ by letting $N (F)$ be in the kernel. Let
$\delta$ be the modular quasicharacter on $B (F)$.
Recall that we are identifying the $\GL_r$ weight lattice
with $\mathbb{Z}^r$. With the usual inner product $\langle\;,\;\rangle$
on $\mathbb{Z}^r$, the simple roots have length $\sqrt{2}$ and the inner product is invariant
under the action of the Weyl group. As before, $\rho$ is the Weyl vector, half the sum of
the positive roots.  With these normalization of the inner product 
\begin{equation}
\label{eq:cablus}
\delta (\varpi^{\lambda}) = q^{- \langle 2 \rho, \lambda \rangle} .
\end{equation}

The {\textit{unramified principal series representations}} are the
representations of $G (F)$ induced from these characters. Thus $I
(\mathbf{z})$ denotes the space of functions $f$ on $G (F)$ that are smooth
(locally constant) and that satisfy $f (b g) = (\delta^{1 / 2}
\chi_{\mathbf{z}}) (b) f (g)$ for $b \in B (F)$. The representation
$\pi_{\mathbf{z}}$ of $G (F)$ on this space is by right translation, so
$(\pi_{\mathbf{z}}(g)f)(x)=f(xg)$.
The representation $\pi (\mathbf{z})$ is irreducible if $\mathbf{z}$ is 
in general position, and if it is irreducible, then
$I (w\mathbf{z}) \cong I (\mathbf{z})$ for $w \in W$, due to
the existence of intertwining integrals $\mathcal{A}^{\mathbf{z}}_w : I
(\mathbf{z}) \longrightarrow I (w\mathbf{z})$
({\cite{CasselmanAdmissible,CasselmanSpherical}}). These are defined by
\[ \mathcal{A}^{\mathbf{z}}_w f (g) = \int_{N \cap w N^- w^{- 1}} f (w^{- 1}
   n g) \, d n, \]
where $N^-$ is the group of unipotent lower triangular matrices. The integral
is only convergent for $\mathbf{z}$ in an open subset of $\hat{T}$, but can
be extended by analytic continuation to the regular elements of $\hat{T}$.

We will follow the approach of Brubaker, Bump and 
Friedberg~{\cite{BrubakerBumpFriedbergRogawski}},
where the action of the Hecke algebra by Demazure-Lusztig operators is related
to the Iwahori fixed vectors in the spherical model. The proof relies on
Casselman~{\cite{CasselmanSpherical}} Theorems~3.4 and~3.1. The results of
Ion~{\cite{IonNonsymmetric}} are also closely related.

There is an error in {\cite{BrubakerBumpFriedbergRogawski}}, where Proposition~1 is only
correct as stated if $\lambda$ is antidominant. Moreover for the application
the constant $c (\lambda)$ in the Proposition needs to be evaluated. See
Propositions~\ref{prop:conjid} and~\ref{prop:intsupport} and 
Remark~\ref{rem:ymincase}.

We will write $g = (g_{ij})$ as a matrix, and if $I \subset \{1, 2, \cdots,
r\}$ we will write $a_I = a_I (g)$ for the minor of $k$ formed with entries
taken from the columns in $I$, and the last $|I|$ rows of $k$. For example if
$n = 4$ and $I = \{1, 3\}$ then
\[ a_I (g) = \left| \begin{array}{cc}
     g_{31} & g_{33}\\
     g_{41} & g_{43}
   \end{array} \right| . \]
If $1\leqslant m\leqslant r$,
we will denote by $I_m$ the particular subset $\{r - m + 1, r - m + 2, \cdots,
r\}$.

\begin{lemma}
  \label{lem:plucker}
  Let $g \in \GL_r (F)$. A necessary and sufficient
  condition for $g\in B J$ is that for $1\leqslant m\leqslant r$,
  and for every $m$ element subset $I$ of $\{1,2,\cdots,r\}$ except $I_m$, we have
  $a_{I_m} (g)^{- 1} a_I (g) \in \mathfrak{p}$.
  Assuming this, $g \in NJ$ if and only if $a_{I_m} (g) \in
  \mathfrak{o}^{\times}$, so that $a_I (g) \in \mathfrak{p}$ when $|I| = m$
  but $I \neq I_m$.
\end{lemma}

\begin{proof}
  We leave this to the reader.
\end{proof}

Let $\lambda$ be a dominant weight. Let
\[ J_{\lambda} = \left\{ (k_{i j}) \in K | \text{$k_{i j} \in \mathfrak{p}^{1 + \lambda_j - \lambda_i}$ when $i > j$} \right\}.
\]
This is clearly a subgroup of~$J$.

\begin{lemma}
  \label{lem:jlamindex} Assume that $\lambda$ is dominant. Then
  the index of $J_{\lambda}$ in $J$ is $q^{\langle
  \lambda, 2 \rho \rangle}$.
\end{lemma}

\begin{proof}
  This group has an Iwahori factorization ({\cite{CasselmanAdmissible}}
	Section~1.4 generalizing \cite{IwahoriMatsumoto} Theorem~2.1), which means 
	in this particular case that we can write
  $J_{\lambda} = (N_- (\mathfrak{p}) \cap J_{\lambda}) B (\mathfrak{o})$. From
  this we can compute
  \begin{equation}
    \label{eq:jlamindex} [J : J_{\lambda}] = [N_- (\mathfrak{p}) : N_-
    (\mathfrak{p}) \cap J_{\lambda}] = \prod_{i > j} q^{\lambda_i - \lambda_j}
    = q^{\langle \lambda, 2 \rho \rangle} .
  \qedhere
  \end{equation}
\end{proof}

\begin{proposition}
  \label{prop:conjid}Suppose that $\lambda$ is antidominant. Then
  \[ \label{eq:pikchar} \{k \in K| \varpi^{- \lambda} k \varpi^{\lambda} \in
     BJ\} = J_{- \lambda} . \]
  This is thus a subgroup of $J$, of index $q^{\langle - \lambda, 2 \rho
  \rangle}$.
\end{proposition}

\begin{proof}
  Since $BJ$ is not a group, it is not \textit{a priori} clear that $\{k \in
  K| \varpi^{- \lambda} k \varpi^{\lambda} \in BJ\}$ is a group. We will show
  that if $k \in K$ and $\varpi^{- \lambda} k \varpi^{\lambda} \in BJ$, then
  $\varpi^{- \lambda} k \varpi^{\lambda} \in NJ$. 
  Let $a_I(k)$ be the minors as above. We will also denote $S (\lambda, I) = \sum_{j
  \in I} \lambda_j$. We have
  \[ a_I  (\varpi^{- \lambda} k \varpi^{\lambda}) = \varpi^{S (\lambda, I) - S
     (\lambda, I_m)} a_I (k) . \]
  If $I \neq I_m$, by Lemma~\ref{lem:plucker} we see that $a_I  (\varpi^{-
  \lambda} k \varpi^{\lambda}) \in \mathfrak{p}$ because $\varpi^{- \lambda} k
  \varpi^{\lambda} \in BJ$. 
  Let $N_I = S (\lambda, I_m) - S (\lambda, I)$. Then 
  $N_I \geqslant 0$ since $\lambda$ is antidominant, so 
  $a_I (k) = \varpi^{N_I} a_I  (\varpi^{- \lambda} k
  \varpi^{\lambda})$. In particular $a_I (k) \in \mathfrak{p}$.
  
  With $m$ fixed, the minors $a_I (k)$ with $|I| = m$ are coprime, since 
  $k\in \GL_r (\mathfrak{o})$. We have shown that if $I \neq I_m$ then
  $a_I (k) \in \mathfrak{p}$ and it follows that $a_{I_m} (k)$ is a unit. But
  $a_{I_m}  (\varpi^{- \lambda} k \varpi^{\lambda}) = a_{I_m} (k)$ and by
  Lemma~\ref{lem:plucker}, we learn that 
  $\varpi^{- \lambda} k \varpi^{\lambda} \in NJ$.
  
  Since $a_I (k) \in \mathfrak{p}$ for all $I \neq I_m$, the set $\{k \in K|
  \varpi^{- \lambda} k \varpi^{\lambda} \in BJ\}$ is contained in $J$, and we
  have proved that if $g$ is in this set then $a_I (g)$ is in $\mathfrak{p}^{1
  + S (\lambda, I_m) - S (\lambda, I)}$. Furthermore, its diagonal entries are
  units. This implies that it is contained in $J_{- \lambda}$. The index is
  given by Lemma~\ref{lem:jlamindex}.
\end{proof}

Let $\phi^{\mathbf{z}}_w$ be the basis of Iwahori fixed vectors in $I
(\mathbf{z})$ defined by
\[ \phi^{\mathbf{z}}_w (b w' k) = \left\{\begin{array}{ll}
     \delta^{1 / 2} \chi_{\mathbf{z}} (b) & \text{if $w = w'$,}\\
     0 & \text{otherwise},
   \end{array}\right. \]
for $w, w' \in W$, $b \in B (F)$ and $k \in J$. Note that every element of $G
(F)$ is in $B (F) w' J$ for a unique $w' \in W$, so this definition makes
sense. Let $\phi_{\circ}^{\mathbf{z}}$ be the $K$-fixed vector, so
\[ \phi^{\mathbf{z}}_{\circ} (b k) = \delta^{1 / 2} \chi_{\mathbf{z}} (b),
   \qquad b \in B (F), \; k \in K. \]

We normalize the Haar integral so that $J$ has volume $1$, and $K$ has
volume
\[\sum_{w\in W}q^{\ell(w)}=\prod_{i=1}^{r}\frac{q^i-1}{q-1}.\]
Let $\mathcal{S}^{\mathbf{z}} : I (\mathbf{z}) \longrightarrow \mathbb{C}$
be the spherical functional:
\[ \mathcal{S}^{\mathbf{z}} (\phi) = \int_K \phi (k) \, d k \]
and for $w \in W$ let $\sigma_w^{\mathbf{z}} : G (F) \longrightarrow
\mathbb{C}$ be the function
\[ \sigma_w^{\mathbf{z}} (g) =\mathcal{S}^{\mathbf{z}} (\pi_{\mathbf{z}}
   (g) \phi^{\mathbf{z}}_w) = \int_K \phi^{\mathbf{z}}_w (k g) \, d k. \]
The {\textit{spherical function}} is defined by
\[ \sigma_{\circ}^{\mathbf{z}} (g) =\mathcal{S}^{\mathbf{z}}
   (\pi_{\mathbf{z}} (g) \phi^{\mathbf{z}}_{\circ}) = \int_K
   \phi^{\mathbf{z}}_{\circ} (k g) \, d k. \]
It is both left and right invariant by $K$.

\begin{theorem}
  \label{thm:demrecurse}Let $w \in W$ and let $s_i$ be a simple reflection
  such that $s_i w > w$. Then
  \begin{equation}
  \label{eq:sigmademrec}
  \sigma_{s_i w}^{\mathbf{z}} (g) =\mathcal{L}_{i,q} \sigma_w^{\mathbf{z}} (g) .
  \end{equation}
  Assume that $\lambda$ is dominant. Then
  \begin{equation}
    \label{eq:sigmaw} \sigma_w^{\mathbf{z}} (\varpi^{w_0 \lambda}) = q^{-
    \langle \lambda, \rho \rangle} \mathcal{L}_{w, q} \mathbf{z}^{w_0
    \lambda} .
  \end{equation}
\end{theorem}

\begin{proof}
  The recursion (\ref{eq:sigmademrec})
  is proved in {\cite{BrubakerBumpFriedbergRogawski}}, Theorem~1, based
  on results of Casselman~\cite{CasselmanSpherical}. It is equivalent to
  {\cite{IonNonsymmetric}}, Proposition~5.8.

  Now (\ref{eq:sigmaw}) follows by induction on $\ell (w)$ provided
  we first establish the base case where $w = 1$. 
  Let $\mu = w_0 \lambda$, so $\mu$ is antidominant.
  Then
  \[ \sigma_1^{\mathbf{z}} (\varpi^{w_0 \lambda}) = \int_K \phi_1^{\mathbf{z}} (k
     \varpi^{\mu}) \, d k = \delta^{1 / 2} \chi (\varpi^{\mu}) \int_K \phi_1^{\mathbf{z}}
     (\varpi^{- \mu} k \varpi^{\mu}) \, d k. \]
  We are normalizing the Haar measure so the volume of $J$ is $1$. By
  Proposition~\ref{prop:conjid}, the integrand is nonzero if and only if $\varpi^{-
  \mu} k \varpi^{\mu} \in B J$ and since $\mu$ is antidominant, the condition
  for this is that $\varpi^{- \mu} k \varpi^{\mu} \in J_{- w_0 \lambda}$,
  so the contribution is the volume $q^{2 \langle \rho, -w_0 \lambda \rangle} =
  q^{2 \langle \rho, \lambda \rangle}$, which is the reciprocal of
  $[J:J_{-w_0(\lambda)}]$. On the other hand by (\ref{eq:cablus}) we have
  $\delta^{1 / 2}
  (\varpi^{\mu}) = q^{- \langle \rho, \mu \rangle} = q^{\langle \rho, \lambda
  \rangle}$, while $\chi_{\mathbf{z}} (\varpi^{\mu}) =\mathbf{z}^{\mu}
  =\mathbf{z}^{w_0 \lambda}$. Thus we have the {\textit{base case}}
  \[ \sigma_1^{\mathbf{z}} (\varpi^{w_0 \lambda}) = q^{- \langle \lambda,
     \rho \rangle} \mathbf{z}^{w_0 \lambda} . \]
  This proves (\ref{eq:sigmaw}) when $w = 1$. 
\end{proof}

We may now give a proof of the Macdonald formula.

\begin{theorem}[Macdonald
{\cite{MacdonaldTata,MacdonaldBook,CasselmanSpherical}}]
  Let $\lambda$ be dominant. Then
  \begin{equation}
  \label{eq:macdonald}
  \sigma_{\circ}^{\mathbf{z}} (\varpi^{\lambda}) = q^{| \Phi^+ |} q^{-
     \langle \lambda, \rho \rangle} R_{\lambda} (\mathbf{z}; q^{- 1}) . 
  \end{equation}
\end{theorem}

\begin{proof}
  Since $\sigma_{\circ}^{\mathbf{z}}$ is $K$-bi-invariant, 
  \[\sigma^\mathbf{z}_\circ(\varpi^\lambda)=
  \sigma^\mathbf{z}_\circ(w_0\varpi^\lambda w_0^{-1})=
  \sigma^\mathbf{z}_\circ(\varpi^{w_0\lambda})\]
  Summing (\ref{eq:sigmaw}) over $w \in W$ and applying
  Proposition~\ref{prop:sphericaldem} and Proposition~\ref{prop:rlamns} gives
  \[ \sigma_{\circ}^{\mathbf{z}} (\varpi^{w_0\lambda}) = q^{- \langle \lambda,
     \rho \rangle} \Omega \left( \prod_{\alpha \in \Phi^+} (1 -
     q\mathbf{z}^{- \alpha}) \mathbf{z}^{w_0 \lambda} \right) = q^{|
     \Phi^+ |} q^{- \langle \lambda, \rho \rangle} R_{\lambda} (\mathbf{z};
     q^{- 1}) . \qedhere\]
\end{proof}

\begin{proposition}
\label{prop:intsupport}
  Suppose that $\lambda \in \Lambda$ and let $y \in W$ be such that 
  $y(\lambda)$ is antidominant. Suppose that $k \in K$ such that $k
  \varpi^{\lambda} \in B y J$. Then $\varpi^{- y (\lambda)} k \varpi^{\lambda}
  \in N K$.
\end{proposition}

\begin{proof}
  If $\alpha$ is a root, we will denote by $x_{\alpha} : F \longrightarrow
  \GL_r$ the one-parameter subgroup corresponding to $\alpha$. We may
  write $k \varpi^{\lambda} = \varpi^{\mu} n y j$ where $\mu \in \Lambda$, $n
  \in N$ and $j \in J$. We will show that $\mu = y (\lambda)$.
  
  Using the Iwahori factorization 
	{\cite{IwahoriMatsumoto, CasselmanAdmissible}}, 
  we may write $j$ as a product of an element
  $t$ of $T (\mathfrak{o})$ and elements of the form $x_{\alpha} (u_{\alpha})$
  where $\alpha \in \Phi$ and $u_{\alpha} \in \mathfrak{o}$. (Actually
  $u_{\alpha} \in \mathfrak{p}$ if $\alpha$ is a negative root.) We
  may choose the order of these factors so that the factors where
  $y(\alpha)$ is a positive root are to the left of those for which
  it is a negative root. If $y (\alpha) \in \Phi^+$ we may conjugate 
  $x_{\alpha} (u_{\alpha})$ by $y$ and
  absorb it into $n$, so we may write
  \[ k \varpi^{\lambda} = \varpi^{\mu} n y j_0 t, \qquad t \in T
     (\mathfrak{o}), n \in N, j_0 = 
     \prod_{\substack{ \alpha \in \Phi\\ y (\alpha) \in \Phi^- }} 
     x_{\alpha} (u_{\alpha}) . \]
  Now with $y (\alpha) \in \Phi^-$ we have
  \[ \varpi^{\lambda} x_{\alpha} (u_{\alpha}) \varpi^{- \lambda} = x_{\alpha}
     (\varpi^{\langle \alpha^{\vee}, \lambda \rangle} u_{\alpha}) = x_{\alpha}
     (\varpi^{\langle y (\alpha)^{\vee}, y (\lambda) \rangle} u_{\alpha}) \in
     K, \]
  because $y (\lambda)$ is antidominant, so $\langle y (\alpha)^{\vee}, y
  (\lambda) \rangle \geqslant 0$. Thus $\varpi^{\lambda} j_0 \varpi^{-
  \lambda} \in K$. Thus 
  $k = \varpi^{\mu - y (\lambda)} n' y \varpi^{\lambda} j_0 \varpi^{- \lambda} t$ 
  where $n'\in N$, which implies that $\varpi^{\mu - y (\lambda)} n' \in K$. This 
  is possible only if $\mu = y (\lambda)$. Thus $\varpi^{- y
  (\lambda)} k \varpi^{\lambda} \in N K$.
\end{proof}

\begin{remark}
  \label{rem:ymincase}
  If we assume that $y$ is minimal such that $y (\lambda)$ is antidominant, a
  refinement of this proof shows that $\varpi^{- y (\lambda)} k
  \varpi^{\lambda} \in N y J$. We do not need this, so we do not prove it.
\end{remark}

\begin{theorem}
  Let $\lambda \in \Lambda$ and let $y \in W$ be the element of minimal length
  such that $y (\lambda)$ is antidominant. Then for any $w \in W$ we have
  \begin{equation}
  \label{eq:sigmafineval} \sigma^{\mathbf{z}}_w (\varpi^{\lambda}) = q^{- \langle \rho,
     \lambda \rangle} \tau^{y (\lambda)}_{w, y} (\mathbf{z} ; q) .
  \end{equation}
  Equivalently,
  \begin{equation}
  \label{eq:equivsigmaeval}
   \sigma^{\mathbf{z}}_w (\varpi^{\lambda}) = q^{-\langle\rho,\lambda\rangle}q^{|\Phi^+|}
   \tau^{-y(\lambda)}_{ww_0,yw_0}(\mathbf{z}^{-1},q^{-1}).
   \end{equation}
\end{theorem}

\begin{proof}
  Note that (\ref{eq:sigmafineval}) and (\ref{eq:equivsigmaeval}) are equivalent
  by (\ref{eq:tauinvolution}). We will prove (\ref{eq:sigmafineval}).
  First we show that
  \begin{equation}
    \label{eq:tensigma} \sigma_y^{\mathbf{z}} (\varpi^{\lambda}) = v (\lambda, q) q^{\ell
    (y)} \mathbf{z}^{y (\lambda)}
  \end{equation}
  where the constant $v (\lambda, q)$ is independent of $\mathbf{z}$; later
  we will evaluate this constant. We have
  \[ \sigma_y^{\mathbf{z}} (\varpi^{\lambda}) = \int_K \phi_y^{\mathbf{z}} (k \varpi^{\lambda}) \, d k.
  \]
  If $\phi_y^{\mathbf{z}} (k \varpi^{\lambda})$ is nonzero, then $k \varpi^{\lambda} \in B
  y J$ and so by Proposition~\ref{prop:intsupport} we have $\varpi^{- y
  (\lambda)} k \varpi^{\lambda} \in N K$. It follows that
  \[ \phi_y^{\mathbf{z}} (k \varpi^{\lambda}) = \delta^{1 / 2} \chi_{\mathbf{z}}
     (\varpi^{y (\lambda)}) \phi_y^{\mathbf{z}} (\varpi^{- y (\lambda)} k \varpi^{\lambda})
     = q^{- \langle \rho, y (\lambda) \rangle} \mathbf{z}^{y (\lambda)} . \]
  Therefore
  \[ \sigma_y^{\mathbf{z}} (\varpi^{\lambda}) = q^{- \langle \rho, y (\lambda) \rangle}
     \mathbf{z}^{y (\lambda)} \cdot V \]
  where $V$ is the volume of $\{ k \in K|k \varpi^{\lambda} \in ByJ \}$. 
  Therefore (\ref{eq:tensigma}) is valid with
  \[ v (\lambda, q) = q^{-\langle \rho, y (\lambda) \rangle}\,q^{-\ell (y)}\,V.
  \]
  
  Now using Theorem~\ref{thm:demrecurse} and the definition of $\tau$
  (implicit in Theorem~\ref{thm:taurecurse}) we obtain
  \begin{equation}
    \label{eq:wteneval} \sigma_w^{\mathbf{z}} (\varpi^{\lambda}) = v (\lambda, q) \tau^{y
    (\lambda)}_{w, y} (\mathbf{z}; q)
  \end{equation}
  for all $w \in W$. We may now evaluate the constant $v (\lambda, q)$ as
  follows. Since $y (\lambda)$ is antidominant, summing this identity over all
  $w \in W$ and using (\ref{eq:dualrform}) gives
  \[ \sigma_{\circ}^{\mathbf{z}} (\varpi^{\lambda}) = v (\lambda, q) \, q^{| \Phi^+ |}
     R_{w_0 y (\lambda)} (\mathbf{z}; q^{- 1}) . \]
  On the other hand since $w_0 y (\lambda)$ is dominant, by
  (\ref{eq:macdonald}) we have
  \[ \sigma_{\circ}^{\mathbf{z}} (\varpi^{w_0 y (\lambda)}) = \, q^{| \Phi^+ |} q^{-
     \langle \lambda, \rho \rangle} R_{w_0 y (\lambda)} (\mathbf{z}; q^{-
     1}) . \]
  Now the spherical function $\sigma_{\circ}^{\mathbf{z}}$ is $K$-bi-invariant, so
  $\sigma_{\circ}^{\mathbf{z}} (\varpi^{\lambda}) = \sigma^{\mathbf{z}}_\circ (\varpi^{w_0 y (\lambda)})$
  and so
  \[ v (\lambda, \rho) = q^{- \langle \lambda, \rho \rangle} . \]
  Combining this with (\ref{eq:wteneval}) proves (\ref{eq:sigmafineval}).
\end{proof}

We would now like to compare this with the partition functions of
the colored models.

\begin{theorem}
\label{thm:mainthm}
Let $\lambda$ be a weight, and let $y\in W$ be minimal such
that $y(\lambda)$ is antidominant. Then
\begin{equation}
\label{eq:comparison}
\sigma_w^{\mathbf{z}}(\varpi^\lambda)=
q^{-\langle\rho,\lambda\rangle}q^{|\Phi^+|}
Z(\mathfrak{S}_{-y(\lambda),ww_0\mathbf{c}_0,yw_0\mathbf{c}_0}
(\mathbf{z}^{-1};q^{-1})).
\end{equation}
\end{theorem}

\begin{proof}
The weight $y(\lambda)$ here is antidominant, whereas in
Theorem~\ref{thm:demeval}, $\lambda$ is dominant. Therefore
we must use (\ref{eq:equivsigmaeval}) instead of (\ref{eq:sigmafineval})
for this comparison. With this in mind, the result follows
from comparing (\ref{eq:equivsigmaeval}) with Theorem~\ref{thm:demeval}.
\end{proof}

\begin{corollary}
\label{cor:maincor}
Let $w\in W$ and $g\in G$. Then there exists a solvable
lattice model whose partition function equals $\sigma^{\mathbf{z}}_w(g)$.
\end{corollary}

\begin{proof}
Since $\phi_w$ is constant on the double cosets for $K\backslash G/J$ it
is sufficient to prove this for a set of double coset representatives,
which we take to be $\varpi^\lambda$ with $\lambda\in\Lambda$. The
result thus follows from Theorem~\ref{thm:mainthm}.
\end{proof}
\bibliographystyle{habbrv}
\bibliography{colored-bosons}

\end{document}